\documentclass[12pt]{amsart}

\usepackage[top=1in, bottom=1in, left=1in, right=1in]{geometry}

\makeatletter
\renewcommand\part{\@startsection{part}{2}%
	\z@{0.5\linespacing\@plus1\linespacing}{\linespacing}%
	{\normalfont\large\scshape\bfseries\centering}}
\makeatother

\usepackage[utf8]{inputenc}
\usepackage{amssymb,amsthm,amsmath,amsxtra,bbm}
\usepackage{enumitem,fancyhdr}
\usepackage{mathtools}
\usepackage{expl3}

\usepackage{times}
\usepackage{bm}
\usepackage{bbm}
\usepackage{scrextend}
\usepackage{mathrsfs}

\usepackage[table,dvipsnames]{xcolor}
\usepackage{color}
\definecolor{red}{rgb}{1,0,0}
\definecolor{orange}{rgb}{0.7,0.3,0}
\definecolor{blue}{rgb}{0,.3,.7}
\definecolor{green}{rgb}{0,.6,.4}

\PassOptionsToPackage{hyphens}{url}\usepackage[colorlinks=true, linkcolor=NavyBlue, citecolor=teal, urlcolor=gray]{hyperref}   

\newcommand{\NN}{\mathbb{N}}
\newcommand{\ZZ}{\mathbb{Z}}

\newcommand{\RR}{\mathbb{R}}
\newcommand{\CC}{\mathbb{C}}
\newcommand{\PP}{\mathbb{P}}
\newcommand{\EE}{\mathbb{E}}

\newcommand{\boldalpha}{\boldsymbol{\alpha}}
\newcommand{\bolddelta}{\boldsymbol{\delta}}
\newcommand{\boldrho}{\boldsymbol{\rho}}

\newcommand{\dee}{\,\mathrm{d}}

\newcommand{\one}{\mathbbm{1}}

\newcommand{\BC}{\mathcal{B}}
\newcommand{\EC}{\mathcal{E}}

\newcommand{\RS}{\mathscr{R}}

\newcommand{\prob}[1]{\PP[ #1 ]}

\newcommand{\bigoo}[1]{O\!\left(#1 \right)}

\renewcommand{\hat}{\widehat}
\renewcommand{\tilde}{\widetilde}
\newcommand{\abs}[1]{\left\lvert #1 \right\rvert}

\newcommand{\floor}[1]{\left\lfloor #1 \right\rfloor}

\newcommand{\riemint}[4]{\int_{#1}^{#2}\! #3 \, \mathrm{d}#4}

\newcommand{\maxnorm}[1]{\left\| #1 \right\|_{\infty}}

\renewcommand{\Re}{\text{Re}}

\newcommand{\dtv}{\mathrm{d}_{\mathrm{TV}}}
\newcommand{\dw}{\mathrm{d}_{\mathrm W}}

\renewcommand{\le}{\leqslant}
\renewcommand{\ge}{\geqslant}

\renewcommand{\epsilon}{\varepsilon}

\newcommand{\pextra}{P_{\mathtt{extra}}}
\newcommand{\parratia}{P_{\mathtt{Arratia}}}

\newcommand{\simplex}{\Delta^{k-1}}

\theoremstyle{plain}
\newtheorem{cor}{Corollary}[section]
\newtheorem{lem}[cor]{Lemma}
\newtheorem{prop}[cor]{Proposition}

\newtheorem{theo}{Theorem}

\theoremstyle{definition}
\newtheorem{example}{Example}
\newtheorem{deff}[cor]{Definition}

\theoremstyle{remark}

\newtheorem*{rems}{Remark}
\newtheorem*{remss}{Remarks}

\makeatletter
\ams@newcommand{\multiint}[1]{\DOTSI\protect\MultiIntegral{#1}}
\renewcommand{\MultiIntegral}[1]{%
	\edef\ints@c{\noexpand\intop
		\ifnum#1=\z@\noexpand\intdots@\else\noexpand\intkern@\fi
		\replicate{#1-2}{\noexpand\intop\noexpand\intkern@}%
		\noexpand\intop
		\noexpand\ilimits@
	}%
	\futurelet\@let@token\ints@a
}
\makeatother
\ExplSyntaxOn
\cs_new:Npn \replicate #1 #2 { \prg_replicate:nn { #1 } { #2 } }
\ExplSyntaxOff

\usepackage[labelfont=bf, width=\linewidth]{caption}

\usepackage{hypcap}   
\usepackage{bookmark} 

\DeclareMathOperator{\vard}{Var}

\DeclareMathOperator{\Dir}{Dir}

\usepackage{subfiles}
\usepackage{xr} 

\numberwithin{equation}{section}

\newcommand{\nn}{\nonumber \\}

\newcommand{\bg}{\big}
\newcommand{\bgg}{\Big}
\newcommand{\bggg}{\bigg}
\newcommand{\bgggg}{\Bigg}

\newcommand{\Rnt}{R_{\mathrm{NT}}}
\newcommand{\Rprob}{R_{\mathrm{PD}}}
\newcommand{\ECnt}{\EC_{\mathrm{NT}}}
\newcommand{\ECprob}{\EC_{\mathrm{PD}}}

\newcommand{\Bnt}{\mathcal{B}_{\mathrm{NT}}}
\newcommand{\Bprob}{\mathcal{B}_{\mathrm{PD}}}

\title{On Arratia's coupling and the Dirichlet law for the factors of a random integer}

\author{Tony Haddad}
\address{D\'epartement de math\'ematiques et de statistique\\
	Universit\'e de Montr\'eal\\
	CP 6128 succ. Centre-Ville\\
	Montr\'eal, QC H3C 3J7\\
	Canada}
\email{{\tt tony.haddad@umontreal.ca}}

\author{Dimitris Koukoulopoulos}
\address{D\'epartement de math\'ematiques et de statistique\\
	Universit\'e de Montr\'eal\\
	CP 6128 succ. Centre-Ville\\
	Montr\'eal, QC H3C 3J7\\
	Canada}
\email{{\tt dimitris.koukoulopoulos@umontreal.ca}}

\date{\today}

\begin{document}
	
	\begin{abstract}
		Let $x\ge2$, let $N_x$ be an integer chosen uniformly at random from the set $\ZZ \cap [1, x]$, and let $(V_1, V_2, \ldots)$ be a Poisson--Dirichlet process of parameter $1$. We prove that there exists a coupling of these two random objects such that 
		\[
		\EE\, \sum_{i \ge 1} \bg|\log P_i- V_i\log x\bg| \asymp 1,
		\]
		where the implied constants are absolute and $N_x = P_1P_2 \cdots$ is the unique factorization of $N_x$ into primes or ones with the $P_i$'s being non-increasing. This establishes a 2002 conjecture of Arratia, who constructed a coupling for which the left-hand side in the above estimate is $\ll\log\!\log x$, and who also proved that the left-hand side is $\ge 1-o(1)$ for all couplings. In addition, we use our refined coupling to give a probabilistic proof of the Dirichlet law for the average distribution of the integer factorization into $k$ parts proved in 2023 by Leung and we improve on its error term.
	\end{abstract}

	\maketitle

	\section{Introduction}
	\label{sec:intro}
	
	Let $N_x$ be an integer chosen uniformly at random from the set $\ZZ \cap [1, x]$. We may then factor it uniquely as $N_x = P_1P_2 \cdots$ with the $P_i$'s forming a non-increasing sequence of primes or ones. In 1972, Billingsley \cite{Bill72} showed that, for any fixed positive integer $r$, the joint distribution of the random vector 
	\[
		\bggg(\frac{\log P_1}{\log x}, \ldots, \frac{\log P_r}{\log x}\bggg)
	\]
	converges in distribution as $x \to \infty$ to the first $r$ components of the Poisson--Dirichlet distribution (of parameter 1). 
	
	There are many ways to define the Poisson--Dirichlet distribution. One of the most intuitive ones involves a ``stick-breaking'' process that we will use throughout the paper. We start by sampling a sequence of i.i.d.~random variables $(U_i)_{i \ge 1}$ that are all uniformly distributed in $[0, 1]$. We then define the sequence $(L_i)_{i \ge 1}$ in the following way:
	\[
		L_1 \coloneqq U_1 \quad \text{and}\quad  L_j \coloneqq U_j \prod_{i=1}^{j-1} (1-U_i) \quad \text{for $j \ge 2$.}
	\] 
	The distribution of the process $\mathbf L = (L_1, L_2, \ldots)$ is called the \textit{GEM distribution} (of parameter 1), named after Griffiths, Engen and McCloskey (see \cite[Chapter 41]{JKB} for a discussion on the history of this distribution). Lastly, we sort the components of $\mathbf L$ in non-increasing order to create $\mathbf V = (V_1, V_2, \ldots)$. The distribution of this process is the \textit{Poisson--Dirichlet distribution} (of parameter 1)\footnote{The GEM and Poisson--Dirichlet distributions have more general definitions involving typically a parameter $\theta$. In the rest of the paper, we will not be mentioning the parameter since we will always work with $\theta = 1$.}. We note that both $\sum_{i \ge 1} L_i$ and $\sum_{i \ge 1} V_i$ are equal to $1$ almost surely.
	
	In 2000, Tenenbaum \cite{Tenenbaum00} studied the rate of convergence in Billingsley's theorem by providing an asymptotic series for the difference between the cumulative distribution functions of $\bg(\tfrac{\log P_1}{\log x}, \ldots, \tfrac{\log P_r}{\log x}\bg)$ and of $(V_1, \ldots, V_r)$.
	
	Another way to give a quantitative version of Billingsley's result is by constructing a \emph{coupling} of $N_x$ and $\mathbf V$, i.e. a single probability space over which lives copies of $N_x$ and $\mathbf V$, such that the expectation
	\begin{equation}
		\label{eq:expected-l1-dist}
		\EE\sum_{i \ge 1} \bg|\log P_i - V_i\log x\bg|
	\end{equation}
	is bounded by a positive monotone function that is $o(\log x)$ as $x \to \infty$. The random variables $N_x$ and $\mathbf V$ must be strongly correlated in this new probability space to achieve this. Indeed, if, for instance, $P_1$ and $V_1$ were independent, then we would have that $P_1\le x^{1/3}$ and $V_1>2/3$ can happen at the same time with positive probability, which implies $\bg|\log P_1 - V_1\log x\bg| \ge \frac{\log x}{3}$ with positive probability. Hence, a coupling with $N_x$ and $\mathbf V$ being independent (also called a \emph{trivial coupling}) makes \eqref{eq:expected-l1-dist} $\asymp \log x$. 
	
	In 2002, Richard Arratia \cite{Arratia02} constructed a coupling satisfying
	\begin{equation}
		\label{eq:arratia02-result}
		\EE\sum_{i \ge 1} \bg|\log P_i - V_i \log x \bg| \ll \log\!\log x
	\end{equation}
	for all $x \ge 3$. Moreover, he conjectured that there is a coupling for $N_x$ and $\mathbf V$ with the expectation above being $O(1)$. The main goal of this paper is to prove this conjecture:
	
	\begin{theo}
	\label{thm:coupling-ub} 
	There is a coupling of $N_x$ and $\mathbf V$ satisfying
	\[
		\EE\, \sum_{i \ge 1} \bg|\log P_i - V_i\log x\bg| \asymp 1
	\]
	for all $x \ge 1$.
	\end{theo}

	\begin{remss}
	(a) Theorem \ref{thm:coupling-ub} is optimal in the sense that no coupling of $N_x$ and $\mathbf V$ can make this expectation tend to $0$ as $x \to \infty$. This follows directly from the triangle inequality:
	\[
		\EE \sum_{i \ge 1} \bg|\log P_i - V_i\log x\bg| \ge \EE \bgg|\sum_{i \ge 1} \bg(\log P_i - V_i\log x\bg)\bgg| = \EE \bg[\log(x/N_x)\bg] = 1-o(1).
	\]
%
	\medskip
	
	(b) If $X$ and $Y$ are two random variables taking values in a separable metric space $(S,\mathrm{d})$, 
	their {\it Wasserstein distance} equals
	\begin{equation}
		\label{eq:def wasserstein}
		\dw(X,Y)  \coloneqq \inf \EE\bg[\mathrm{d}(X,Y)\bg],
	\end{equation}
	where the infimum is taken over all possible couplings of $X$ and $Y$. Taking $S$ to be the $\ell^1$ space of real valued sequences with its associated metric, and using remark (a) above, we find that Theorem \ref{thm:coupling-ub} is equivalent to the statement that $\dw(\mathbf{P}_x, \mathbf V) \asymp1/\log x$, where $\mathbf{P}_x \coloneqq \bg(\frac{\log P_1}{\log x}, \frac{\log P_2}{\log x}, \ldots\bg)$.
	
	\medskip
	
	(c) The coupling we construct contains a copy of the random process $(N_x)_{x \ge 1}$. This fact might be useful in applications of Theorem \ref{thm:coupling-ub}. 
	\medskip
	
	(d) Let $\sigma$ be a random permutation uniformly distributed in the permutation group $S_n$. It is well known that the factorization into primes of $N_x$ and the decomposition into disjoint cycles of $\sigma$ share similar statistics when $n \approx \log x$.  In 2006, Arratia, Barbour and Tavaré \cite{ABT06} have proved that there exists a coupling between $\sigma$ and $\mathbf V$ such that
	\begin{equation}
		\label{eq:ABT permutation thm}
		\EE \sum_{i \ge 1} |C_i - n V_i| \sim \frac{\log n}{4},
	\end{equation}
	with $C_i$ being the number of cycles of length $i$ in $\sigma$. They showed that \eqref{eq:ABT permutation thm} was optimal by using the inequality $|C_i - nV_i| \ge \|nV_i\|$ where $\left\|\cdot\right\|$ is the distance to the closest integer, and computing $\EE \sum_{i \ge 1} \|nV_i\|$. This breaks the analogy between primes and permutations since Theorem \ref{thm:coupling-ub} and \eqref{eq:ABT permutation thm} are not of the same order of magnitude when $n$ is replaced by $\log x$. The main reason why it is possible to get a better result in Theorem \ref{thm:coupling-ub} is because the set $\{\log p : p \text{ primes}\}$ have much shorter gaps around $\log x$ than the gaps of $\ZZ$ around $n$.  
	\end{remss}

	\subsection{Application to the distribution of factorizations of random integers}
	
	We now give some applications of Theorem \ref{thm:coupling-ub} to the theory of divisors of integers.
	
	Let $\Delta^{\infty}$ be the set of sequences $(v_i)_{i \ge 1}$ of non-negative reals satisfying $\sum_{i \ge 1} v_i = 1$, and let $\psi\colon \Delta^{\infty} \to \RR$ be a $K$-Lipschitz function with respect to the $\ell^1$-metric on $\Delta^\infty$. Since the diameter of $\Delta^\infty$ is $2$, the image of $\psi$ is contained in an interval of length $2K$. Consequently, $\psi$ is bounded, and we have $\EE[\one_{N_x = 1}\cdot \psi(\mathbf V)] \ll_{\psi} 1/x$.  By Theorem \ref{thm:coupling-ub}, it follows that
	\begin{equation}
		\label{eq:dual wasserstein}
		\EE\bgg[\one_{N_x \ge 2}\cdot \psi\bg(\tfrac{\log P_1}{\log N_x}, \tfrac{\log P_2}{\log N_x}, \ldots\bg)\bgg] = \EE\bgg[\psi\bg(V_1, V_2, \ldots\bg)\bgg] + O_\psi(1/\log x) .
	\end{equation}
	Using this simple fact, we obtain: 
	\begin{cor}
		Let $\rho(n) \coloneqq \min\{d|n : d\ge \sqrt n\}$. Then there is a constant $c \in (1/2,1)$ such that 
		\[
			\sum_{n\le x} \log \rho(n) = cx\log x + O(x) \qquad (x \ge 1) .
		\]
	\end{cor}
	
	\begin{proof}
	Let $\psi(v_1, v_2, \ldots) \coloneqq \inf \big( \mathcal D(v_1, v_2, \ldots) \cap [\frac{1}{2}, 1]\big)$ with $\mathcal D(v_1, v_2, \ldots) \coloneqq \{\sum_{i \in B} v_i : B\subseteq \mathbb N\}$. It is easy to check that this is a $1$-Lipschitz function. Note that the left-hand side of \eqref{eq:dual wasserstein} with this definition for $\psi$ is exactly equal to $\frac{1}{\floor x} \sum_{2 \le n \le x} \frac{\log \rho(n)}{\log n}$. With \eqref{eq:dual wasserstein} and partial summation, we prove the corollary with $c = \EE[\psi(\mathbf V)]$.
	\end{proof}

	\medskip

	For the next application, we use the coupling again to extract information about integer factorizations. However, to obtain the desired result, we require much more than just the expectation in Theorem \ref{thm:coupling-ub}.
	
	Let $\simplex$ be the set of $k$-tuples $\boldalpha=(\alpha_1, \ldots, \alpha_k) \in \RR^k_{\ge 0}$ satisfying $\alpha_1+\cdots+\alpha_k=1$. We also need to define a special class of functions:
	
	\begin{deff}[The class of functions $\mathcal F_k(\boldalpha)$]\label{dfn:Fktheta}
			Given $k\in\ZZ_{\ge2}$ and $\boldalpha\in \simplex$, let $\mathcal F_k(\boldalpha)$ be the set of functions $f \colon \NN^{k} \to \RR_{ \ge 0}$ satisfying the following three properties:
		\begin{enumerate}[label=(\alph*)]
			\item For any fixed positive integer $n$, the function $f(\mathbf d)$ is a probability mass function over all vectors $\mathbf d \in \NN^k$ satisfying $d_1\cdots d_k = n$, i.e. $\sum_{d_1\cdots d_k = n} f(d_1, \ldots, d_k) = 1$ for all $n \in \NN$.
			
			\item Whenever $\mathbf d$ satisfies 
			\[
			d_i \coloneqq \begin{cases}
				p &\text{if $i = j$} \\
				1 &\text{if $i \ne j$}.
			\end{cases}
			\]
			for some $1 \le j \le k$ and prime $p$, then $f(\mathbf d) = \alpha_j$.
			
			\item The function $f$ is \emph{multiplicative}, i.e. for any vectors $\mathbf a, \mathbf b \in \NN^k$ such that $(a_1\cdots a_k, b_1 \cdots b_k) = 1$, we have the property
			\[
			f(a_1b_1, \ldots, a_kb_k) = f(a_1, \ldots, a_k) \cdot f(b_1, \ldots, b_k).
			\]
		\end{enumerate}
	\end{deff}
	
	\begin{rems}
			Let $\omega(n)$ denote the number of distinct prime factors of $n$. If $d_1\cdots d_k$ is a square-free number and $f\in\mathcal{F}_k (\boldalpha)$, then properties (b) and (c) of the definition above imply that
		\begin{equation}
			\label{eq:f on squarefrees}
			f(d_1,\dots,d_k) = \prod_{j=1}^k \alpha_j^{\omega(d_j)}.
		\end{equation}
	\end{rems}
	
	We will use the class of functions $\mathcal{F}_k(\boldalpha)$ to define certain ``random factorizations'' into $k$ parts of a random integer. Specifically, let us fix $f\in\mathcal{F}_k(\boldalpha)$ and $x\ge1$. We then define a {\it random $k$-factorization corresponding to $f$} to be a random vector $\mathbf{D}_{f, x}=(D_{f,x,1},\dots,D_{f,x,k})$ taking values on $\NN^k$ and satisfying the formula\footnote{Strictly speaking, we only need property (a) of Definition \ref{dfn:Fktheta} to define $\mathbf{D}_{f,x}$. But we will also need the other two properties when proving Theorem \ref{thm:fact into k parts} below.}
	\begin{equation}
		\label{eq:prob mass function f}
		\PP\bgg[D_{f,x,i} = d_i \ \, \forall i \le k\, \bgg|\, N_x = n\bgg] = f(d_1, \ldots, d_k)
	\end{equation}
 for all $n\in\NN$ and all $k$-tuples $(d_1,\dots,d_k)\in\NN^k$ with $d_1\cdots d_k=n$. 
	
	Here are three examples of such random factorizations.
	
	\begin{example}[Uniform sampling]
		\label{ex:uniform sampling}
		Let $f_{\mathrm U(k)}(d_1, \ldots, d_k) \coloneqq \tau_k(d_1\cdots d_k)^{-1}$ with $\tau_k(n)$ being the number of $k$-factorizations of $n$. Then $f_{\mathrm U(k)} \in \mathcal F_k(\frac{1}{k}, \ldots, \frac{1}{k})$. If $f_{\mathrm U(k)}$ is seen as a probability mass function as in \eqref{eq:prob mass function f}, then we are sampling $D_{f_{\mathrm U(k)}, x, 1}\cdots D_{f_{\mathrm U(k)}, x, k}$ uniformly among all $k$-factorizations of $N_x$.
	\end{example}

	\begin{example}[Recursive sampling]
		\label{ex:recursive sampling}
		Let $f_{\mathrm R(k)}(d_1, \ldots, d_k) \coloneqq \prod_{j=1}^{k-1} \tau(d_j \cdots d_k)^{-1}$ with $\tau(n)$ being the number of divisors of $n$. Then $f_{\mathrm R(k)} \in \mathcal F_k(\frac{1}{2}, \frac{1}{4}, \ldots, \frac{1}{2^{k-1}}, \frac{1}{2^{k-1}})$. One way to realize this random $k$-factorization is by first sampling uniformly a divisor $D_{f_{\mathrm R(k)}, x, 1}$ of $N_x$. Then, for all $j < k$, we recursively sample $D_{f_{\mathrm R(k)}, x, j}$ uniformly among the divisors of $\frac{N_x}{D_{f_{\mathrm R(k)}, x, 1} \cdots D_{f_{\mathrm R(k)}, x, j-1}}$.
	\end{example}

	\begin{example}[Multinomial sampling]
		\label{ex:multinomial sampling}
		For any fixed $\boldalpha \in \simplex$, let $f_{\mathrm M(\boldalpha)}(d_1, \ldots, d_k) \coloneqq \prod_{i=1}^{k} \alpha_i^{\Omega(d_i)} \cdot \prod_{p|n} \binom{\nu_p(d_1\cdots d_k)}{\nu_p(d_1), \ldots, \nu_p(d_k)}$ with $\nu_p(d)$ being the $p$-valuation of $d$ and $\Omega(d)$ being the number of prime factors of $d$ counted with multiplicity. The function $f_{\mathrm M(\boldalpha)}$ is in $\mathcal F_k(\boldalpha)$. This sampling can be understood as considering a sequence of i.i.d.~random variables $(B_i)_{i \ge 1}$ satisfying $\PP[B_i = j] = \alpha_j$ and constructing the $k$-factorization $D_{f_{\mathrm M(\boldalpha)}, x,1}\cdots D_{f_{\mathrm M(\boldalpha)}, x,k} = N_x$ as 
		\[
		D_{f_{\mathrm M(\boldalpha)}, x,j} \coloneqq \prod_{\substack{i\ge 1:\ B_i = j}} P_i,
		\]
		where $P_1 P_2 \cdots$ is the prime factorization of $N_x$ as before. With this definition, the vectors $(\nu_p(D_{f_{\mathrm M(\boldalpha)}, x,1}), \cdots, \nu_p(D_{f_{\mathrm M(\boldalpha)}, x,k}))$ all follow multinomial distributions for every prime $p$.
	\end{example}
	
	Starting with the pioneering work of Deshouillers, Dress, and Tenenbaum \cite{DDT79}, many different people have studied the distribution of $\mathbf D_{f, x}$ for various choices of $f \in \mathcal F_k(\boldalpha)$, and they proved that it converges on a logarithmic scale to a certain {\it Dirichlet distribution}. Recall that if $\boldalpha \in \simplex$ satisfies $\alpha_i > 0$ for all $i$, we say that a $\simplex$-valued random vector $\mathbf Z$ follows the Dirichlet distribution $\Dir(\boldalpha)$ if 
	\[
		\PP\bg[Z_i \le u_i\ \, \forall i<k\bg] = F_{\boldalpha}(\mathbf u) \coloneqq \prod_{i=1}^{k} \Gamma(\alpha_i)^{-1} \mathop{\int\cdots\int}\limits_{\substack{0 \le t_i \le u_i \ \forall i < k \\ t_1 + \cdots + t_{k-1} \le 1}} \prod_{i=1}^k t_i^{\alpha_i - 1} \dee t_1\cdots \dee t_{k-1}
	\]
	With this notation, most known results on $\mathbf D_{f, x}$ can be stated in the following form: uniformly for $x \ge 2$ and $\mathbf u \in [0, 1]^{k-1}$, we have 
	\begin{equation}
		\label{eq:previous result on k fact}
		\PP\bg[D_{f, x, i} \le N_x^{u_i}\ \, \forall i < k\bg] = F_{\boldalpha}(\mathbf u) + O_{\boldalpha}\bgg((\log x)^{-\min\{\alpha_1, \ldots, \alpha_{k}\}}\bgg)
	\end{equation}

	As we indicated above, the first result of this type was obtained in 1979 by Deshouillers, Dress, and Tenenbaum \cite{DDT79}, who proved \eqref{eq:previous result on k fact} for $f = f_{\mathrm U(2)}$ (which is also equal to $f_{\mathrm R(2)}$; both cases involve sampling a divisor of $N_x$ uniformly). For higher values of $k$, Nyandwi and Smati \cite{NyandwiSmati13} extended this result in 2013 by proving \eqref{eq:previous result on k fact} for $f = f_{\mathrm U(3)}$, and La Bretèche and Tenenbaum \cite{delaBretecheTenenbaum16} further proved it for $f = f_{\mathrm R(3)}$ in 2016. In 2023, Leung \cite{Ken22} demonstrated that \eqref{eq:previous result on k fact} applies to $f = f_{\mathrm U(k)}$ for all $k \ge 2$. Leung also claimed that his proof could be adapted to any $f$ in a more general class of $k$-dimensional multiplicative functions than $\mathcal F_k(\boldalpha)$ (for a proof, see \cite{KenWeb24}). This more general class allows the quantities $f(1, \ldots, 1, p, 1, \ldots, 1)$, with $p$ being at the $j$\textsuperscript{th} coordinate, to equal $\alpha_j$ {\it on average} rather than pointwise.
	
	Returning to 2007, Bareikis and Manstavi\v{c}ius \cite{BareikisManstavicius07} generalized the result of Deshouillers, Dress, and Tenenbaum with an improvement on the error term. They showed that if $f$ belongs to the more general class of two-dimensional multiplicative functions containing $\mathcal F_2(\alpha_1, \alpha_2)$, as described in Leung's claim, then \eqref{eq:previous result on k fact} holds with the improved error term 
	\[
		\PP\bg[D_{f, x, 1} \le N_x^u\bg] = F_{(\alpha_1,\alpha_2)}(u) + O_{\alpha_1, \alpha_2}\bgg((1 + u\log x)^{-\alpha_1}(1 + (1-u)\log x)^{-\alpha_2}\bgg).
	\]
	
	All results mentioned above rely on Fourier-analytic techniques, such as the Landau--Selberg--Delange method, as their main ingredients to get to their results. 
	
	\medskip
	
	We have a new approach to this problem: since the size of the divisors on a logarithmic scale of any integer is entirely determined by the size of its prime factors, one might expect that the distribution of $\mathbf D_{f, x}$ can be accessed via a quantitative form of Billingsley's theorem. Indeed, using the coupling from Theorem \ref{thm:coupling-ub}, we establish the following result, which proves \eqref{eq:previous result on k fact} in full generality with a Bareikis-Manstavi\v{c}ius type error-term. 
	
	\begin{theo}[Dirichlet law for the factorization into $k$ parts]
		\label{thm:fact into k parts}
		Let $k \ge 2$, let $\boldalpha \in \simplex$ be fixed with $\alpha_i > 0$ for all $i$, and let $f\in\mathcal{F}_k(\boldalpha)$. In addition, let $x>1$ and let $\mathbf{D}_{f, x}$ be the random $k$-factorization corresponding to $f$.
		
		For any $\mathbf u \in [0, 1]^{k-1}$ with at least one $i \in \{1, \ldots, k-1\}$ with $u_i \ne 1$, we have 
		\[
		\PP\bg[D_{f, x,i} \le N_x^{u_i}\ \, \forall i < k\bg] = F_{\boldalpha}(\mathbf u)
		+ O\bgggg(\sum_{\substack{1 \le i < k \\ u_i \ne 1}} \frac{1}{(1+u_i\log x)^{1-\alpha_i}(1+(1-u_i)\log x)^{\alpha_i}}\bgggg) \,;
		\]
		the implied constant in the big-Oh is completely uniform in all parameters.
	\end{theo}

	\begin{remss} (a) When $u\in[0,1/2]$, we have $(1+u_i\log x)^{1-\alpha_i}(1+(1-u_i)\log x)^{\alpha_i}\ge (0.5\log x)^{\alpha_i}$. Similarly, when $u\in[1/2,1]$, we have $(1+u_i\log x)^{1-\alpha_i}(1+(1-u_i)\log x)^{\alpha_i}\ge (0.5\log x)^{1-\alpha_i}$. We thus find that the expression in the big-Oh in Theorem \ref{thm:fact into k parts} is $\le (0.5\log x)^{-\min\{\alpha_1,\dots,\alpha_k,1-\alpha_1,\dots,1-\alpha_k\} }$ uniformly in $\mathbf u\in[0,1]^{k-1}$. Since $1-\alpha_j\ge\alpha_i$ for all $i\neq j$ by our assumption that $\alpha_1+\cdots+\alpha_k=1$, we conclude that the error term in Theorem \ref{thm:fact into k parts} is $\ll_{\boldalpha} (\log x)^{-\min\{\alpha_1,\dots,\alpha_k\}}$, thus recovering Leung's estimate \eqref{eq:previous result on k fact} when $f$ lies in the class $\mathcal{F}_k(\boldalpha)$.
		
	(b) If we relied solely on Arratia's original coupling from \cite{Arratia02}, which leads to \eqref{eq:arratia02-result}, our approach would yield the same Theorem \ref{thm:fact into k parts} with each instance of $\log x$ in the error term being replaced by $\frac{\log x}{\log_2 x}$. This adjustment would still offer an improvement over Leung's result if all $u_i$ values are sufficiently far from $0$ or $1$. We provide further explanations about this point in the remarks following Lemma \ref{lem:bound of prob of ECprob}.
	\end{remss}
	
	The proof of Theorem \ref{thm:fact into k parts} is based on a 1987 result of Donnelly and Tavaré \cite{DonnellyTavare87}, who proved the following probabilistic version of Leung's theorem: If $\mathbf V = (V_1, V_2, \ldots)$ is a Poisson--Dirichlet process, and $(C_i)_{i \ge 1}$ is a sequence of i.i.d.~random variables that is independent of $\mathbf V$, supported on $\{1, \ldots, k\}$ and satisfying $\PP[C_i = j] = \alpha_j$ for all $j$, then
	\begin{equation}
		\label{eq:donnelly tavare 87}
		\bgggg(\sum_{i:\ C_i = 1} V_i, \ldots, \sum_{i:\ C_i = k} V_i\bgggg)
	\end{equation}
	follows exactly the Dirichlet distribution $\Dir(\boldalpha)$. In 1998, Arratia \cite{Arratia98} used this result to show that $\PP\bg[D_{f_{\mathrm U(2)}, x, 1} \le N_x^u\bg]$ is $F_{(\frac{1}{2}, \frac{1}{2})}(u) + o(1)$ as $x \to \infty$ with probabilistic methods. We use the coupling to bridge between the distribution of \eqref{eq:donnelly tavare 87} and Theorem \ref{thm:fact into k parts} and get an explicit error term.
	
	\begin{rems}
		Here is a brief heuristic about the shape of the error term we obtain in Theorem \ref{thm:fact into k parts}. Suppose that $N_x \ne 1$ and let $\bolddelta_{f, x} \coloneqq (\frac{\log D_{f, x, 1}}{\log N_x}, \ldots, \frac{\log D_{f, x, k}}{\log N_x})$.  There exists a coupling between $\bolddelta_{f, x}$ and the random vector \eqref{eq:donnelly tavare 87} such that their distance is of typical size $\asymp \frac{1}{\log x}$. For each $j$, the marginal distribution of the $j$\textsuperscript{th} component of $\Dir(\boldalpha)$ is $\mathrm{Beta}(\alpha_j, 1-\alpha_j)$, which is why we get the error term of Theorem \ref{thm:fact into k parts}.
	\end{rems}

	\subsection{Structure of the paper}
	
	We have organized the paper in two main parts. 
	
Part \ref{part:Arratia} contains the proof of Theorem \ref{thm:coupling-ub}. It is divided as follows:
	\begin{itemize}
		\item In Section \ref{sec:coupling}, we present the coupling implicit in Theorem \ref{thm:coupling-ub} and present a proof of the latter as a corollary of four key results (Lemmas \ref{lem:coupling ineq}-\ref{lem:properties of Theta} and Proposition \ref{prop:dtv of M and N}). Lemma \ref{lem:properties of N} is simple and proven right away.
		
		\item In Section \ref{sec:proof of Lemma coupling ineq}, we prove Lemma \ref{lem:coupling ineq}.

		\item In Section \ref{sec:GEM}, we explain another way to realize a GEM process that was presented by Arratia in \cite{Arratia02}, and we use it to prove Lemma \ref{lem:properties of Theta}. This alternative way of describing a GEM process will be also key in proving Proposition \ref{prop:dtv of M and N}.
		
		\item Sections \ref{sec:integer friendly J}, \ref{sec:different J and J*} and \ref{sec:dtv} are reserved to prove Proposition \ref{prop:dtv of M and N}. 
	\end{itemize}  

Finally, Part \ref{part:DDT} contains the proof of Theorem \ref{thm:fact into k parts} and it is organized in the following way:
\begin{itemize}
	\item In Section \ref{sec:Donnelly-Tavare}, we present the argument due to Donnelly--Tavar\'e showing that random $k$-partitions of the components of a Poisson--Dirichlet process are distributed according to a Dirichlet law. 
	
	\item In Section \ref{sec:coupling for factorizations}, we use the coupling of Section \ref{sec:coupling} to construct a coupling of the random $k$-factorization $\mathbf D_{f, x}$ and an analogous $k$-partition of the components of the Poisson--Dirichlet process. We then use this coupling to reduce Theorem \ref{thm:fact into k parts} to estimating two boundary events, one involving number-theoretic objects and the other one using pure probabilistic objects.

	\item In Section \ref{sec:boundary event NT}, we prove the necessary estimate for the number-theoretic boundary event, and in Section \ref{sec:boundary event Prob} we show the analogous probabilistic estimate.
\end{itemize}  

\begin{rems}
	The readers interested only in Theorem \ref{thm:fact into k parts} do not need to go beyond Section \ref{sec:coupling} in Part \ref{part:Arratia}.
\end{rems}

	\subsection{Notation}
	We let $\log_j$ denote the $j$-iteration of the natural logarithm, meaning that $\log_1=\log$ and $\log_j=\log\circ \log_{j-1}$ for $j\ge2$. 
	
	Throughout the paper, the letter $p$ is reserved for prime numbers and the letter $n$ is reserved for natural numbers, unless stated otherwise. Given such $p$ and $n$, we write $\nu_p(n)$ for the $p$-adic valuation of $n$, that is to say the largest integer $v\ge0$ such that $p^v|n$. In addition, we write $\omega(n)$ for the number of distinct prime factors of $n$.
	
	Moreover, for $n\in\NN$, we let $s(n)$ denote its largest square-full divisor. Also, we let $n^\flat \coloneqq n/s(n)$ and we note that $n^\flat$ is always square-free and co-prime to $s(n)$.

	We write $\pi(x)$ for the number of primes $\le x$. We shall also use heavily Chebyshev's function $\theta(x) \coloneqq \sum_{p\le x}\log p$.
	
	To describe various estimates, we use Vinogradov's notation $f(x) \ll g(x)$ or Landau's notation $f(x) = \bigoo{g(x)}$ to mean that $\abs{f(x)} \le C\cdot g(x)$ for a positive constant $C$. If $C$ depends on a parameter $\alpha$,  we write $f(x) \ll_\alpha g(x)$ or $f(x) = O_\alpha(g(x))$. If two positive functions $f, g$ have the same order of magnitude in the sense that $f(x) \ll g(x) \ll f(x)$, then we write $f(x) \asymp g(x)$. 
	
	If $P$ is some proposition, then the indicator function $\one _P$ will be equal to 1 if $P$ is true and $0$ if $P$ is false.

	\section*{Acknowledgements}
	The authors would like to thank Matilde Lal\'in, Sun-Kai Leung and G\'erald Tenenbaum for their helpful comments on the paper.
	
	TH is supported by the Courtois Chair II in fundamental research. DK is supported by the Courtois Chair II in fundamental research, by the Natural Sciences and Engineering Research Council of Canada (RGPIN-2018-05699 and RGPIN-2024-05850) and by the Fonds de recherche du Qu\'ebec - Nature et technologies (2022-PR-300951 and 2025-PR-345672).

	\newpage
	
	\clearpage
	\thispagestyle{fancy}
	\fancyhf{} 
	\renewcommand{\headrulewidth}{0cm}
	\lhead[{\scriptsize \thepage}]{}
	\rhead[]{{\scriptsize\thepage}}
	\part{Sharpening Arratia's coupling}\label{part:Arratia}

	\section{The coupling}
	\label{sec:coupling}
	
	In this section, we describe the coupling behind Theorem \ref{thm:coupling-ub}. To construct it, we begin with an ambient probability space $\Omega$ containing the following objects:
	\begin{itemize}
		\item a GEM process $\mathbf L = (L_1, L_2, \ldots)$;
		
		\item three mutually independent random variables $U'_1, U'_2$ and $U'_3$ that are also independent from $\mathbf L$, and which are uniformly distributed in the open interval $(0, 1)$.
	\end{itemize}  
	We shall extract an integer $N_x$ and a Poisson--Dirichlet process $\mathbf V$ as deterministic functions of these random objects. Extracting $\mathbf V$ is done by sorting the components of $\mathbf L$ in non-increasing order. The extraction of $N_x$ is more complicated, and we need to introduce additional notation to describe it.
	
	Let $(\lambda_j)_{j \ge 0}$ be the increasing sequence of positive real numbers defined by $\lambda_0 \coloneqq e^{-\gamma}$ and 
	\[
	\lambda_j \coloneqq \exp\!\bggg(\!-\gamma+ \sum_{i \le j} \frac{1}{v_iq_i}\bggg)\quad \text{for $j \ge 1$}
	\] 
	with  $\gamma$ being the Euler-Mascheroni constant and $q_j = p_j^{v_j}$ being the $j$\textsuperscript{th} smallest prime power, i.e., $(q_j)_{j \ge 1}$ is the sequence $2, 3, 2^2, 5, 7, 2^3, 3^2,\dots$ Note that
	\begin{equation}
		\label{eq:lambda_j}
		\lambda_j  = \log q_j +O(1/(\log q_j)^2)
	\end{equation}
	by a stronger version of Mertens's estimate (Proposition \ref{prop:mertens}).
		
	Moreover, we have that
	\begin{equation}
		\label{eq:q_j}
		q_{j+1}  = q_j+O(q_j/(\log q_j)^3).
	\end{equation}	
	Indeed, the Prime Number Theorem (Proposition \ref{pnt}) implies that $\theta(q_j+Cq_j/(\log q_j)^3) -\theta(q_j)>0$ if $C$ is large enough, whence $q_{j+1}\le q_j+Cq_j/(\log q_j)^3$, as needed. 
	
	Next, we define the step-function $h:\RR_{>0} \to \RR_{>0}$ by 
	\begin{equation}
		\label{eq:h(t) dfn}
		h(t) \coloneqq \sum_{j \ge 1} (\log q_j) \cdot \one_{\lambda_{j-1} < t \le \lambda_j}
	\end{equation}
	In particular, \eqref{eq:lambda_j} and \eqref{eq:q_j} imply that, if $r(t) \coloneqq |h(t) - t|$, then 
		\begin{equation}
			\label{eq:h(t) estimate}
			r(t) \ll \min\{t, t^{-2}\} \quad \text{for all $t > 0$.}
		\end{equation}
		
	Using the above notation, here is how to extract $N_x$ from $\mathbf L$, $U'_1$, $U'_2$ and $U'_3$:
		\begin{enumerate}[topsep=1em, itemsep=1em]
			\item Construct the sequence of random prime powers or ones $(Q_i)_{i \ge 1}$ by letting $Q_i \coloneqq e^{h(L_i\log x)}$. Note that we have $Q_i = 1$ whenever $L_i \le \frac{e^{-\gamma}}{\log x}$ so there are only finitely many $Q_i$'s that are prime powers.
			
			\item Define the random integer $J_x \coloneqq \prod_{j \ge 2} Q_j$.
			
			\item Define the \emph{extra prime} $\pextra$ as the smallest element in the set $\{1\}\cup\{\text{primes}\}$ that would satisfy $\theta(\pextra)\ge U_1'\theta(x/J_x)$, where $\theta(y)=\sum_{p\le y}\log p$ is Chebyshev's function. (In particular, we have $\pextra=1$ when $J_x>x/2$; otherwise, $\pextra$ is a prime $\le x/J_x$.)
			
			\item Let $\mu_x$ be the probability measure induced by the random variable  $M_x\coloneqq J_x\pextra$, and let $\nu_x$ be the uniform counting measure on $\ZZ\cap[1,x]$. Then, by Lemma \ref{lem:coupling-dtv} and our assumption that $U_2'$ and $U_3'$ exist on $\Omega$, there exists a random variable $N_x$ on $\Omega$ such that:
			
			\smallskip 
			
					\begin{itemize}
						\item $N_x$ is uniformly distributed on $\ZZ\cap[1,x]$,
						\item $\PP\bg[M_x\neq N_x\bg]=\dtv(\mu_x,\nu_x)$
					\end{itemize}
			\smallskip
			with $\dtv$ being the total variation distance defined in \eqref{eq:dtv dfn}.
		\end{enumerate}

	This completes the definition of our coupling, since the space $\Omega$ contains a Poisson--Dirichlet process $\mathbf V$ and also a random variable $N_x$ with distribution $\nu_x$. 
	
	In Section \ref{sec:coupling:reduction}, we show how to use the coupling to prove Theorem \ref{thm:coupling-ub}. Lastly, in Section \ref{sec:coupling:remarks}, we make some technical remarks on the coupling.

	\subsection{Reducing Theorem \ref{thm:coupling-ub} to three lemmas and a proposition}\label{sec:coupling:reduction}
	
	With the following four key results, we directly get Theorem \ref{thm:coupling-ub}. Recall that $s(n)$ denotes the largest square-full divisor of the integer $n$. In addition, let 
	\begin{equation}
		\label{eq:def of Theta_x}
	\Theta_x \coloneqq \sum_{i \ge 1} r(V_i\log x) .
	\end{equation}
	
	\begin{lem}[The $\ell^1$ distance within the coupling]
		\label{lem:coupling ineq}
		When $M_x = N_x$, we have the inequality
		\[
			\sum_{i \ge 1} |\log P_i - V_i\log x| \le \log(x/N_x) + 2\cdot \log s(N_x) + 2\cdot \Theta_x.
		\]
	\end{lem}

	\begin{lem}[Properties of $N_x$]
		\label{lem:properties of N}
			Fix $\alpha\in[0,1)$ and $\beta\in[0,1/2)$. Uniformly over $x\ge1$, we have 
			\[
			\EE\bg[(x/N_x)^\alpha s(N_x)^\beta\bg] \ll_{\alpha,\beta} 1.
			\]
	\end{lem}
	
	\begin{proof} We must show that
	\[
	S\coloneqq	\sum_{n \le x} (x/n)^\alpha s(n)^\beta\ll_{\alpha,\beta} x.
	\]
	Indeed, if we let $b=s(n)$ and $a=n^\flat=n/b$, then
	\[
	S\le \sum_{\substack{b\le x \\ b\ \text{square-full} }} x^\alpha b^{\beta-\alpha} \sum_{a\le x/b} a^{-\alpha} 
	\ll_\alpha \sum_{\substack{b\le x \\ b\ \text{square-full} }} x^\alpha b^{\beta-\alpha} (x/b)^{1-\alpha} 
	\le x \sum_{b\ \text{square-full} } b^{\beta-1},
	\]
	where we used our assumption that $\alpha<1$. 
	Since we have assumed that $\beta<1/2$, the last sum over $b$ converges, thus completing the proof.
	\end{proof}
	
	\begin{lem}[Properties of $\Theta_x$]
		\label{lem:properties of Theta}
			Fix $\alpha\ge 0$. Uniformly over $x\ge1$, we have 
			\[
			\EE\bg[e^{\alpha \Theta_x}\bg] \ll_{\alpha} 1.
			\]
	\end{lem}

	\begin{prop}[Total variation distance between $M_x$ and $N_x$]
		\label{prop:dtv of M and N}
		For $x \ge 2$, we have 
		\[
			\PP\bg[M_x \ne N_x\bg] \ll \frac{1}{\log x} .
		\]
	\end{prop}
	
	The proof of Lemma \ref{lem:coupling ineq} is given in Section \ref{sec:proof of Lemma coupling ineq}, and the proof of Lemma \ref{lem:properties of Theta} is given in Section \ref{sec:GEM}.  The proof of Proposition \ref{prop:dtv of M and N} is the longest part. We set it up in Sections \ref{sec:integer friendly J} and \ref{sec:different J and J*} to eventually give it in Section \ref{sec:dtv}. Here is how we get Theorem \ref{thm:coupling-ub} with these results:

	\begin{proof}[Proof of Theorem \ref{thm:coupling-ub}]
		Let $S \coloneqq \sum_{i \ge 1} |\log P_i - V_i\log x|$. The proof that $\EE[S] \gg 1$ is explained in the remark after the statement of the theorem in the introduction. For the upper bound, we first note that we always have the trivial bound $S \le 2 \log x$. This bound and Lemma \ref{lem:coupling ineq} gives us
		\[
			S \le \one_{M_x \ne N_x}\cdot (2\log x) + \one_{M_x = N_x}\cdot \bg(\log(x/N_x) + 2\cdot \log s(N_x) + 2\cdot \Theta_x\bg).
		\] 
		Taking expectations on both sides, we get $\EE[S] \ll 1$ with Lemmas \ref{lem:properties of N}-\ref{lem:properties of Theta} and with Proposition \ref{prop:dtv of M and N}.
	\end{proof}

	In fact, if we condition on the event $M_x = N_x$, we can obtain a much stronger bound:
	
	\begin{prop}\label{thm:coupling-ub strong}
		Fix $\alpha\in[0,1/4)$. For $x\ge2$, we have
		\[
		\EE\bgggg[\exp\bgg(\alpha \sum_{i \ge 1} |\log P_i - V_i \log x|\bgg)  \,\bggg|\, M_x = N_x \bgggg] \ll_{\alpha} 1    .
		\]
	\end{prop}
	
	\begin{proof}
	This follows readily by H\" older's inequality and by Lemmas \ref{lem:coupling ineq}-\ref{lem:properties of Theta}.
	\end{proof}

	\subsection{Remarks on the coupling}\label{sec:coupling:remarks}

	(a) As discussed previously, we have $\lambda_{j-1}\approx \lambda_j\approx \log q_j$, and thus $(\log x)L_i \approx \log Q_i$ as long as $L_i$ is not too small. In particular, we expect that $\sum_{i \ge 1} \log Q_i$ would be too close to $\log x$, and thus $\prod_{i\ge1}Q_i$ cannot serve as a proxy of $N_x$. This is the reason we have to delete $Q_1$ from the factors of $J_x$, and we insert instead an extra random prime $\pextra$ conveniently chosen so that $J_x\pextra$ has a distribution close to $\nu_x$. 
	
	As we already remarked, we have $\pextra=1$ if, and only if, $J_x > x/2$ (which happens rarely); otherwise, $\pextra$ is a prime $\le x/J_x$. As a matter of fact, for all $j\in\ZZ\cap[1,x/2]$, we have
	\begin{equation}
		\label{eq:distribution of Pextra}
		\PP\bg[\pextra = p\,|\, J_x = j\bg] =  \frac{\one_{p\le x/j}\cdot\log p}{\theta(x/j)}.
	\end{equation}
	This is the crucial property that will allow us to show that $M_x=J_x\pextra$ is close to being uniformly distributed.
	
	\medskip
	
	(b) The coupling we defined above is a modification of Arratia's coupling in \cite{Arratia02}. Some of the differences in our definition are purely aesthetic. The one major difference is within step (3), which is the whole reason why we obtain a stronger bound than \eqref{eq:arratia02-result}. The construction of Arratia's extra prime $\parratia$ had a different distribution which satisfied 
	\begin{equation}
		\label{eq:extraprime-arratia}
		\PP\bg[\parratia = p\,|\, J_x=j\bg] = \frac{1}{1+\pi(x/j)} 
	\end{equation}
	for all $j\le x$ and all $p \in \{1\}\cup\{\text{primes}\le x/j\}$. It is possible to get the inequality in Lemma \ref{lem:coupling ineq} with Arratia's original coupling. However, it would be impossible to get a version of Proposition \ref{prop:dtv of M and N} with a bound better than $\frac{\log_2x}{\log x}$.

	\section{The $\ell^1$ distance within the coupling}\label{sec:proof of Lemma coupling ineq}
	
	In this section we establish Lemma \ref{lem:coupling ineq}. 
	We need the following rearrangement inequality.
	
	\begin{lem}[Rearrangement inequality]
		\label{lem:rearrangement-ineq}
		For any two non-increasing sequences $(x_i)_{i \ge 1}$ and $(y_i)_{i \ge 1}$ of real numbers, and for any two permutations $\sigma, \rho \colon \NN \to \NN$, we have 
		\[
		\sum_{i \ge 1} |x_i - y_i| \le \sum_{i \ge 1} |x_{\sigma(i)} - y_{\rho(i)}|.
		\]
	\end{lem}
	
	\begin{proof} See \cite[Lemma 3.2]{ABT06}. 
	\end{proof}

	\begin{proof}[Proof of Lemma \ref{lem:coupling ineq}]
		Recall the definitions of the sequence of prime powers $(Q_i)_{i \ge 1}$, the extra prime $\pextra$ and $M_x$. We create another sequence of primes or ones $(\tilde P_i)_{i \ge 1}$ in the following way: 
		\begin{itemize}
			\item We set $\tilde P_1 \coloneqq \pextra$.
			
			\item If $i \ge 2$ with $Q_i = 1$, then we set $\tilde P_i \coloneqq 1$.
			
			\item If $i \ge 2$ with $Q_i > 1$, we set $\tilde P_i$ to be the only prime dividing $Q_i$.
		\end{itemize}
		We let $(\hat P_i)_{i\ge1}$ be the sequence $(\tilde{P}_i)_{i\ge1}$ in non-increasing order, and we set
		\[
		\hat M_x \coloneqq \prod_{i\ge1} \tilde{P}_i = \prod_{i\ge1} \hat P_i. 
		\]
		Since $M_x=N_x$, we have $P_i \ge \hat P_i$ for all $i$, because $(\hat P_i)_{i \ge 1}$ is a subsequence of the non-increasing sequence $(P_i)_{i \ge 1}$. Therefore, 
		\[
		\sum_{i \ge 1} |\log P_i - V_i\log x|\le \sum_{i \ge 1} |\log \hat P_i - V_i\log x| + \sum_{i \ge 1} \log(P_i/\hat P_i)
		\]
		We have that $\prod_{i \ge 1} P_i/\hat P_i = M_x/\hat M_x$. Furthermore, the integer $M_x/\hat M_x$ only contains prime factors whose square divides $N_x$. Thus, $M_x/\hat M_x$ divides $s(N_x)$, and
		\begin{equation}
			\label{eq:l1 dist reduction 1}
			\sum_{i \ge 1} \log(P_i/\hat P_i) = \log(M_x/\hat M_x) \le \log s(N_x),
		\end{equation}
		and thus
		\begin{equation}
			\label{eq:l1 dist reduction 2}
			\sum_{i \ge 1} |\log P_i - V_i\log x|\le \sum_{i \ge 1} |\log \hat P_i - V_i\log x| + \log s(N_x).
		\end{equation}
		
		Next, we use the rearrangement inequality (Lemma \ref{lem:rearrangement-ineq}) to find that
		\begin{align*}
			\sum_{i \ge 1} |\log \hat P_i - V_i\log x| &\le |\log \pextra - L_1\log x|+ \sum_{i \ge 2} |\log \tilde P_i - L_i\log x| \\
			&\le |\log \pextra - L_1\log x|+ \sum_{i \ge 2} |\log Q_i - L_i\log x| +  \sum_{i \ge 2} \log(Q_i/\tilde P_i) ,
		\end{align*}
		where we used that $\tilde{P_i}\le Q_i$ for each $i$. Moreover, we have $\prod_{i \ge 2} Q_i/\tilde P_i = M_x/\hat M_x$, so \eqref{eq:l1 dist reduction 1}  implies that
		\[
		\sum_{i \ge 1} |\log \hat P_i - V_i\log x| 
		\le |\log \pextra - L_1\log x|+ \sum_{i \ge 2} |\log Q_i - L_i\log x| +  \log s(N_x) .
		\]					
		Finally, we also note that $\log \pextra = \log M_x - \sum_{i \ge 2} \log Q_i$ and $L_1 = 1-\sum_{i \ge 2} L_i$. Since we have assumed that $M_x=N_x$, we have
		\[
		\label{eq:l1 dist reduction 4}
		|\log \pextra - L_1\log x| \le \log(x/N_x) + \sum_{i \ge 2} |\log Q_i - L_i\log x|.
		\]
		Combining the two above displayed inequalities with \eqref{eq:l1 dist reduction 2}, we conclude that
		\[
		\sum_{i \ge 1} |\log P_i - V_i\log x|\le \log(x/N_x) + 2\sum_{i \ge 2} |\log Q_i - L_i\log x|   + 2\log s(N_x).
		\]
		To complete the proof, recall that $\log Q_i = h(L_i\log x)$ and $r(t) = |h(t) - t|$, whence
		\[
		\label{eq:l1 dist reduction 6}
		\sum_{i \ge 2} |\log Q_i - L_i\log x| = \sum_{i \ge 2} r(L_i\log x) \le \Theta_x. 
		\]
		This proves Lemma \ref{lem:coupling ineq}. 
	\end{proof}

	\section{Another realization of the GEM distribution}
	\label{sec:GEM}
	
	For the proof of Lemma \ref{lem:properties of Theta} and Proposition \ref{prop:dtv of M and N}, we will need to be more precise as to how the GEM process $\mathbf L$ is sampled in the coupling. The construction we present below is also the one used by Arratia \cite{Arratia02}. 
	
	\begin{deff}[The Poisson Process $\RS$]\label{dfn:Poisson Process R}\ 
	\begin{enumerate}[label=(\alph*)]
		\item We denote by $\RS$ the Poisson process on $\RR_{>0}^2$ that has intensity measure $e^{-wy}\, \mathrm dw\, \mathrm dy$. We may assume that the $w$-coordinates of the points of $\RS$ are all distinct since this happens almost surely.
		\item We index the points of $\RS = \bg\{(W_i, Y_i) : i \in\ZZ\bg\}$ according to the following rules:
			\begin{itemize}
				\item $W_i < W_{i+1}$ for all $i\in\ZZ$;
				\item if we let $S_i\coloneqq \sum_{\ell \ge i} Y_\ell$ for all $i\in\ZZ$, then we have $S_1 \le \log x < S_0$.
			\end{itemize} 
	\end{enumerate} 
	\end{deff}
	
	\begin{rems}
		By the Mapping Theorem (Proposition \ref{prel:map}), projecting $\RS$ on the $w$-axis yields a Poisson Process with intensity $\frac{\mathrm dw}{w}$ and, similarly, projecting $\RS$ on its $y$-axis yields a Poisson Process with intensity $\frac{\mathrm dy}{y}$. Therefore, the $w$-coordinates of the points in $\RS$ have almost surely exactly one limit point at $0$ and they are almost surely unbounded. Hence, the indexing $(W_i,Y_i)$ in part (b) of the above definition is well-defined.
	\end{rems}
	
	The following lemma describes the distribution of the point process $S_i$. 
	
	\begin{lem}[Scale-invariant spacing lemma]
		\label{lem:scale-invariant-spacing-lemma}
		The point process $\{S_i \,:\, i \in \ZZ\}$ is a Poisson process on the positive real line with intensity measure $\frac{\mathrm ds}{s}$.
	\end{lem}
	
	\begin{proof}
	See \cite[Lemma 7.1]{ABT06}. 
	\end{proof}
	
	Using this lemma, we have the following description of the GEM distribution. 
	
	\begin{prop}[Arratia, \cite{Arratia02}]
		\label{prop:gemfrompi}
		The process $\bg(1-\frac{S_1}{\log x}, \frac{Y_1}{\log x}, \frac{Y_2}{\log x}, \ldots\bg)$ follows a GEM distribution.
	\end{prop}

	\begin{proof}
		Applying the map $T_x(s) \coloneqq \log_2 x - \log s$ to the points of $\{S_i \,:\, i \in \ZZ\}$ yields a homogeneous Poisson process on the real line with constant rate $1$, by the Mapping Theorem (Proposition \ref{prel:map}) and Lemma \ref{lem:scale-invariant-spacing-lemma}. Furthermore, we have $T_x(S_0) < 0 \le T_x(S_1)$ and $(T_x(S_i))_{i \in \ZZ}$ increasing. Therefore, $T_x(S_1)$ and $T_x(S_{i+1}) - T_x(S_i)$ for all $i \ge 1$ are independent exponential random variables of parameter $1$. If $X$ is a standard exponential random variable, then $1-e^{-X}$ is a uniform random variable in $[0, 1]$. We conclude that $1-\frac{S_1}{\log x},\ \frac{Y_1}{S_1},\ \frac{Y_2}{S_2},\, \ldots$ are independent uniform random variables in $[0, 1]$. The proposition follows by the characterization of the GEM distribution described in the introduction.
	\end{proof}
	
	For the next sections and the proof below, we will assume that the process $\mathbf L = (L_1, L_2, \ldots)$ sampled for our coupling was determined by $\RS$ by defining $L_1 \coloneqq 1-\frac{S_1}{\log x}$ and $L_j \coloneqq \frac{Y_{j-1}}{\log x}$ for $j \ge 2$. In this setting, we now have that 
	\begin{equation}
		\label{eq:J_x alternative definition}
			J_x = \prod_{i \ge 1} e^{h(Y_i)}.
	\end{equation}

	\subsection{Proof of Lemma \ref{lem:properties of Theta}}
		We conclude this section by using the realization of the GEM described here to prove Lemma \ref{lem:properties of Theta}. Note that 
		\begin{equation}
			\label{eq:theta bounded}
			\Theta_x = r(L_1\log x) + \sum_{i \ge 1} r(Y_i) \le \Theta_\infty + O(1)
		\end{equation}
		with $\Theta_\infty \coloneqq \sum_{i \in \ZZ} r(Y_i)$. With Campbell's Theorem (Proposition \ref{prel:campbell}), we directly compute that  
		\begin{equation}
			\label{eq:mgf theta infty}
			\EE[e^{\alpha \Theta_\infty}] = \exp\!\bgggg(\!\int_{0}^\infty \frac{e^{\alpha r(y)} - 1}{y}\, \mathrm dy\bgggg).
		\end{equation}
		This integral is convergent for all fixed $\alpha > 0$ because $e^{\alpha r(y)} -1 \ll_\alpha r(y) \ll\min\{y,y^{-2}\}$. Combining this fact with \eqref{eq:theta bounded} proves that $\EE[e^{\alpha \Theta_x}] \ll_\alpha 1$. We have thus established Lemma \ref{lem:properties of Theta}.

	\section{An integer-friendly version of $J_x$}
	\label{sec:integer friendly J}
	
	Let $\Lambda$ be the von Mangoldt function, that is to say
	\[
	\Lambda(n) = \begin{cases}
		\log p &\text{if $n = p^k$ for some prime power $p^k$,} \\
		0 &\text{otherwise.}
	\end{cases}
	\]
	Recall the Poisson Process $\RS$ given in Definition \ref{dfn:Poisson Process R}. We then define 
	\[
	\RS^* \coloneqq \bgg\{     \bgg( WY/ h(Y), e^{h(Y)}  \bgg): (W,Y)\in \RS,\ Y > e^{-\gamma}\bgg\}. 
	\]
	Without loss of generality, we may assume that the quantities $WY/h(Y)$ with $(W,Y)\in\RS$ and $Y>e^{-\gamma}$ are all distinct. By the Mapping Theorem (Proposition \ref{prel:map}) applied to the map $(w,y)\mapsto (wy/h(y),e^{h(y)})$ on the Poisson process $\RS$ restricted to $\RR_{>0} \times \RR_{>e^{-\gamma}}$, the random set $\RS^*$ is a Poisson process on the space $\RR_{>0}\times \{\text{prime powers}\}$ with mean measure $\mu^*$ satisfying 
	\[
		\mu^*(B \times \{q\}) = \int_{B} \frac{\Lambda(q)}{q^{1+t}} \, \mathrm dt
	\]
	for any $B\subseteq\RR_{>0}$ and any $q\in\NN$. Note that the $w$-coordinates of the points in $\RS$ restricted to $\RR_{>0} \times \RR_{>e^{-\gamma}}$ are now bounded with probability one. For each such realization of the points of $\RS^*$, there is a unique labeling of them as $\{(T^*_i, Q^*_i)\,:\,i \in \ZZ_{\le K}\}$ in such a way that the following properties hold:
	\begin{itemize}
		\item  $T^*_{i-1} <T^*_{i}$ for all $i\in\ZZ_{\le K}$;
		\item $\prod_{i=1}^K  Q_i^*\le x< \prod_{i=0}^K Q_i^*$.
	\end{itemize}
	(Note that $K$ is a random variable and is not fixed.) We then define the random integer
	\[
	 J_x^*\coloneqq \prod_{i=1}^K Q^*_i. 
	 \]
	 One advantage of introducing $J_x^*$ is that the computation for the distribution of $J_x^*$ can be done more easily and precisely than the distribution of $J_x$. We perform this calculation in this section. We will then see in the next section that the probability that $J_x$ is different from $J_x^*$ is small enough to be negligible in the calculation of the total variation distance in Section \ref{sec:dtv}. 
	 
	 This random integer $J^*_x$ was used by Arratia in \cite{Arratia02}, and he estimated its distribution in his Lemma 2. For the sake of completeness, we give a proof for this estimation.	
	
	\begin{lem}[Arratia \cite{Arratia02}]
		\label{lem:probaj}
		For $x \ge 2$ and $1 \le j \le x$, we have
		\[
		\PP\bg[J_x^*=j\bg] = \frac{1}{j\log x}\left(1+\bigoo{\frac{1}{\log x}}\right).
		\]
	\end{lem}

	\begin{proof}
		Let $t>0$ and let $q$ be a prime power, and consider the random variable $N_q(t)$ that counts the number of points $(T^*_i, Q^*_i)$ with $T^*_i > t$ and $Q^*_i = q$. Note that $N_q(t)$ has Poisson distribution with parameter $\frac{\Lambda(q)}{q^{1+t}\log q}$. Moreover, let $I^*_t \coloneqq \prod q^{N_q(t)}$ with the product being over all prime powers $q$. On one hand, when $t$ is fixed, the expectation $\EE[(I^*_t)^{-s}]$ can be expanded as a Dirichlet series whose $i$\textsuperscript{th} coefficient equals $\PP[I^*_t = i]$. This Dirichlet series converges absolutely for $\Re(s)\ge0$. On the other hand, we have
		\[
		\EE[(I^*_t)^{-s}] = \prod_q \EE\bg[q^{-sN_q(t)}\bg]  = \prod_q \exp\!\bgg(\tfrac{\Lambda(q)}{q^{1+t}\log q} (q^{-s}-1)\bgg) = \frac{\zeta(1+t+s)}{\zeta(1+t)}.
		\]
		It follows that $\PP[I^*_t = i] = \frac{i^{-1-t}}{\zeta(1+t)}$ for all $i \ge 1$. All points of $\RS^*$ have distinct $T^*$-coordinates with probability one. In this situation, we have that $J_x^* = j$ if, and only if, there exists exactly one point $(T^*, Q^*) \in \RS^*$ such that $I^*_{T^*} = j$ and $Q^* > \frac{x}{j}$. Thus, we have 
		\[
			\one _{J_x^* = j} = \sum_{(T^*, Q^*) \in \RS^*} \one _{I^*_{T^*} = j \text{ and } Q^{*}> x/j}
		\]
		almost surely. Taking expectations on both sides and using the Mecke equation (Proposition \ref{prel:mecke}), we obtain the distribution of $J_x^*$:
		\begin{equation}
			\label{eq:probaJ}
			\PP\bg[J_x^*=j\bg]
				= \riemint{0}{\infty}{\Bigg(\sum_{q > x/j} \frac{\Lambda(q)}{q^{1+t}}\Bigg)\cdot \frac{j^{-1-t}}{\zeta(1+t)}}{t}.
		\end{equation}
		We want to estimate this integral. Let $S(u) \coloneqq \sum_{q \le u} \frac{\Lambda(q)}{q}$. We have $S(u) = \log u + O(1)$ for $u \ge 1$ by Mertens's estimate \cite[Theorem 3.4(a)]{Koukoulo19}. We use partial summation to get
		\[
			\sum_{q > x/j} \frac{\Lambda(q)}{q^{1+t}} = \riemint{x/j}{\infty}{u^{-t}}{S(u)} = \frac{(x/j)^{-t}}{t}(1+O(t)).
		\]
		for all $t > 0$. By putting this estimate in \eqref{eq:probaJ}, we have
		\[
			\PP[J_x^* = j] = \frac{1}{j} \riemint{0}{\infty}{\frac{x^{-t}(1+O(t))}{t\cdot \zeta(1+t)}}{t} .
		\]
		Since $\zeta(1+t)\ge1$ for all $t>0$, the portion of the integral over $t\ge1$ is $\ll 1/(x\log x)$. On the other hand, if $t\in(0,1]$, we have $1/\zeta(1+t)=t+O(t^2)$. We conclude that
		\[
			\PP[J_x^* = j] = \frac{1}{j} \int_0^1 x^{-t}(1+O(t)) \,\mathrm dt + O\bggg(\frac{1}{jx\log x}\bggg) = \frac{1}{j\log x}\bggg(1+O\bggg(\frac{1}{\log x}\bggg)\bggg).
		\]
		This concludes the proof.
	\end{proof}

	
	\section{When $J_x$ and $J_x^*$ are different}
	\label{sec:different J and J*}

	To prove Proposition \ref{prop:dtv of M and N}, we must get a hold of the distribution of $J_x$. Since we have a good approximation for the distribution of $J_x^*$, it will be enough, for our purposes, to show that the event $\{J_x\neq J_x^*\}$ occurs with low probability. 
	
	For any $t > 0$, consider the random variable 
	\[
		I_t \coloneqq \sum_{(W,Y)\in \RS} Y\cdot \one_{W>t} .
	\]
	We compute below the distribution of $I_t$.
	
	\begin{prop}[Arratia, \cite{Arratia02}]
		\label{prop:icexponential}
		For any fixed $t > 0$, the random variable $I_t$ follows an exponential distribution of parameter $t$, i.e. $\prob{I_t > y} = e^{-ty}$ for all $y > 0$.
	\end{prop}
	
	\begin{proof}
		A direct application of Campbell's Theorem (Proposition \ref{prel:campbell}) implies that $\EE[e^{sI_t}]=t/(t-s)$ for $\Re(s)<t$, which agrees with the moment generating function of an exponential distribution of parameter $t$. This completes the proof.
	\end{proof}
	
	Let $\eta$ be the smallest positive constant satisfying
	\[
		\frac{a}{h(a)}\cdot \frac{h(b)}{b} \le 1 + \frac{\eta}{\min\{a,b\}^2} .
	\]
	for all $a, b > e^{-\gamma}$. Such a constant must exist by \eqref{eq:h(t) estimate}. In addition, let
	\[
	r_0\coloneqq \sup_{y > 0} r(y). 
	\]
	Let us then define the following events: 
	
	\smallskip
	
	\begin{itemize}[itemsep=0.5em]
		\item $\EC_1\coloneqq \bg\{ S_1<\log x-R_x-r_0,\ S_0>\log x+R_x+2r_0\bg\}$, where $R_x\coloneqq \sum_{i\ge2} r(Y_i)$.

		\item $\EC_2$ is the event where $\frac{W_i}{W_0} > 1 + \frac{\eta}{\min\{Y_0,Y_i\}^2}$ for all $i \ge 1$.
		
		\item $\EC_3$ is the event where $\frac{W_0}{W_i} > 1 + \frac{\eta}{\min\{Y_0,Y_i\}^2}$ for all $i \le -1$.
	\end{itemize}
	The variable $R_x$ depends on the value of $x$ since the labeling of points in $\mathscr R$ change as $x$ grows, even if $\mathscr R$ stays fixed.

	\begin{lem}
		\label{lem:decomposition J neq J*}
		For $x > 1$, we have $\EC_1\cap \EC_2\cap \EC_3\subseteq \{ J_x = J_x^*\}$. In particular,
		\[
			\PP\bg[J_x \ne J_x^*\bg] \le \PP[\EC_1^c] + \PP[\EC_1 \cap \EC_2^c] + \PP[\EC_1 \cap \EC_3^c].
		\]
	\end{lem}
	
	\begin{proof} Recall that $J_x=\prod_{i\ge1}e^{h(Y_i)}$ and assume that $\mathcal E_1$ occurs. Then 
		\begin{equation}
			\label{eq:e1 implies J < x}
			\log J_x = \sum_{i=1}^{\infty} h(Y_i) \le \sum_{i \ge 1} \bg(Y_i + r(Y_i)\bg) \le S_1 + R_x  +r_0 < \log x,
		\end{equation}
		since $r(Y_1)\le r_0$. Similarly, we have 
		\[
			\log\!\bg(J_x\cdot e^{h(Y_0)}\bg) = \sum_{i=0}^{\infty} h(Y_i) \ge \sum_{i \ge 0} \bg(Y_i - r(Y_i)\bg) \ge S_0 - R_x -2r_0 > \log x.
		\] 		
		Therefore, $\EC_1$ implies the inequalities 
		\begin{equation}
			\label{eq:J_x inequalities under E_1}
			J_x < x  < J_x e^{h(Y_0)} .
		\end{equation} 
		In particular, we must have $Y_0>e^{-\gamma}$.
		
		Assume now further that $\EC_2$ and $\EC_3$ also occur. We claim that this implies $J_x = J_x^*$. Let $B_+$ be the set of integers $i \ge 1$ such that $Y_i > e^{-\gamma}$, and let $B_-$ be the set of integers $i\le-1$ such that $Y_i>e^{-\gamma}$. 
		By our assumption that $\EC_2\cap\EC_3$ occurs and by the definition of $\eta$, we have
		\begin{equation}
			\label{eq:preserving order 1}
		\frac{W_iY_i}{h(Y_i)} > 		\frac{W_0Y_0}{h(Y_0)} \qquad\text{if}\ i\in B_+
		\end{equation}
		and, similarly, 
		\begin{equation}
			\label{eq:preserving order 2}
		\frac{W_iY_i}{h(Y_i)} <	\frac{W_0Y_0}{h(Y_0)} \qquad\text{if}\ i\in B_-. 
		\end{equation}
		Now, let $K \ge 0$ and $(T_j^*,Q_j^*)$ with $j\in\ZZ_{\le K}$ be the indexing of points of $\RS^*$ given in Section \ref{sec:integer friendly J}, that is to say we have $T_{j-1}^*<T_{j}^*$ for all $j \in \ZZ_{\le K}$, and such that $\prod_{j=0}^K Q_j^*>x \ge \prod_{j=1}^K Q_j^*$. Let $j_0\in\ZZ_{\le K}$ be such that $T_{j_0}^* = W_0Y_0/h(Y_0)$ (which exists because $Y_0>e^{-\gamma}$). Using relations \eqref{eq:preserving order 1} and \eqref{eq:preserving order 2}, we find that 
		\[
		\prod_{j_0 < j \le K} Q_j^*=\prod_{i\in B_+} e^{h(Y_i)} = \prod_{i\ge1} e^{h(Y_i)} = J_x .
		\] 
		On the other hand, we have $\prod_{j_0 \le j \le K} Q_j^*=e^{h(Y_0)} J_x$. Hence, using \eqref{eq:J_x inequalities under E_1}, we find that $\prod_{j_0 < j \le K}Q_j^*<x<\prod_{j_0 \le j \le K}Q_j^*$. In particular, $j_0=0$ and thus $J_x^*=\prod_{j_0 < j \le K} Q_j^*=J_x$, as claimed. 
	\end{proof}

	The following lemma requires a lot more care than in Arratia's paper \cite{Arratia02} since he only needed $\PP[J \ne J^*] \ll \frac{\log_2 x}{\log x}$ to be true for his coupling. 
	
	\begin{lem}
		\label{lem:prob e1c}
		We have $\PP[\EC_1^c] \ll \frac{1}{\log x}$ for $x\ge2$.
	\end{lem}
	
	\begin{proof} We may assume that $x$ is large enough. 
		We have $\EE[e^{R_x}]\le \EE[e^{\Theta_\infty}] \ll 1$ by \eqref{eq:mgf theta infty}. Therefore, the event $R_x > \log_2 x$ occurs with probability $\ll\frac{1}{\log x}$. Using Lemma \ref{lem:scale-invariant-spacing-lemma}, we find that the probability that there exists $i\in\ZZ$ such that $|S_i-\log x|\le 4r_0$ is $O(1/\log x)$. 	In conclusion,
		\[
						\PP[\EC_1^c] \le \sum_{k\in\{0,1\}} \PP\bgg[4r_0<|S_k-\log x|\le R_x+2r_0,\ R_x\le \log_2x \bgg] +O\bggg(\frac{1}{\log x}\bggg)  .
		\]
		Note that if there exists $k \in \ZZ$ with $W_k>(\log_2x)^{-10}$ and $|S_k-\log x|\le \log_2x + 2r_0$, then $I_{(\log_2x)^{-10}}\ge \log x-\log_2x-2r_0$. So, using Proposition \ref{prop:icexponential}, we find that
		\begin{align*}
		\PP\bgg[\exists k: |S_k-\log x|\le \log_2x+2r_0,\ W_k>(\log_2x)^{-10} \bgg] 
			&\le  \PP\bg[I_{(\log_2 x)^{-10}} > \log x-\log_2x-2r_0\bg] \\
			&= \exp\!\bggg(\!-\frac{\log x-\log_2x-2r_0}{(\log_2x)^{10}}\bggg)
		\end{align*}
		for $x$ large enough. Consequently, 
		\[
		\PP[\EC_1^c] 
			\le \sum_{k\in\{0,1\}} \PP\bgg[4r_0<|S_k-\log x|\le R_x+2r_0,\ R_x\le \log_2x,\ W_k\le(\log_2x)^{-10} \bgg] +O\bggg(\frac{1}{\log x}\bggg) .
		\]
		Let $m_0$ be the largest integer such that $2^{m_0-1}r_0\le \log_2x$. Hence, if $2r_0<R_x\le\log_2x$, then there exists a unique integer $m\in[2,m_0]$ such that $2^{m-1}r_0<R_x\le 2^mr_0$, in which case the condition $|S_k-\log x|\le R_x+2r_0$ implies that $|S_k-\log x|\le 2^mr_0 +2r_0\le 2^{m+1}r_0$. We conclude that
			\[
					\PP[\EC_1^c]					
							\le  \mathop{\sum\sum}_{\substack{k\in\{0,1\} \\ 2\le m\le m_0} }
								\PP\bgg[R> 2^{m-1}r_0,\ |S_k-\log x|\le 2^{m+1}r_0,\ W_k \le (\log_2 x)^{-10} \bgg]	
						 +O\bggg(\frac{1}{\log x}\bggg) .
				\]
		Note that $\one_{R_x > 2^{m-1}r_0}\le \frac{4^{1-m}}{r_0^2} \one_{R_x>2r_0}R_x^2$. In addition, if $R_x>2r_0$, then $\sum_{i\ge2}r(Y_i)^2\le r_0R_x<R_x^2/2$, whence $R_x^2\le 4\sum_{i>j\ge2}r(Y_i)r(Y_j)$.
		We conclude that
		\[
									\PP[\EC_1^c]	\le   \mathop{\sum\sum}_{\substack{k\in\{0,1\} \\ 2\le m\le m_0} } 
						 	\frac{4^{2-m}}{r_0^2} \EE\bggg[ \mathop{\sum\sum}_{i,j :\ i>j>k}r(Y_i)r(Y_j) \cdot \one_{|S_k-\log x|\le 2^{m+1}r_0}  \cdot \one_{W_k \le (\log_2x)^{-10}} \bggg]	
						 +O\bggg(\frac{1}{\log x}\bggg) ,
		\]
		Therefore, to complete the proof of the lemma, it is enough to show that
		\begin{equation}
			\label{eq:E(z)}
		E(z) \coloneqq	\EE \bggg[\mathop{\sum\sum\sum}_{\substack{i, j, k \in \ZZ:\ i>j>k}}r(Y_i)r(Y_j) \cdot \one_{|S_k-\log x|\le z} \cdot \one_{W_k \le (\log_2x)^{-10}}  \bgg] \ll \frac{z}{\log x}
		\end{equation}
		uniformly for $z\in[0,4\log_2x]$. 
		
		For the rest of the proof, we fix $z\in[0,4\log_2x]$. 
	Given $t''>t'>t > 0$, let 
	\[
	I_{t,t',t''}\coloneqq \sum_{(W,Y)\in\RS,\ W\in\RR_{>t}\setminus\{t',t''\}} Y. 
	\]
	Since the $W_i$'s are almost surely distinct, we have
	\[
	E(z) \le \EE\bggg[ \mathop{\sum\sum\sum}_{\substack{(W,Y),(W',Y'),(W'',Y'')\in\RS \\ W''>W'>W }} 
		r(Y')r(Y'') \cdot  \one_{|Y+Y'+Y''+I_{W,W',W''}-\log x|\le z} \cdot \one_{W \le (\log_2x)^{-10}}  \bggg]  . 
	\]
	Hence, using the Mecke equation (Proposition \ref{prel:mecke}) and the fact that $I_{t, t', t''}$ has the same distribution as $I_{t}$, we find that
	\[
	E(z) \le \mathop{\int\cdots\int}\limits_{\substack{0<w<w'<w'' \\ w\le (\log_2x)^{-10} \\ y,y',y''\ge0 \\  y+y'+y''\le \log x+z }} r(y')r(y'') 
		\PP\bgg[ \bg|I_w+y+y'+y''-\log x\bg|\le z\bgg] 
		\frac{\dee w\cdots \dee y''}{  e^{wy+w'y'+w''y''}} ,
	\]
	where the integral is sixfold with variables $w,w',w'',y,y',y''$. Proposition \ref{prop:icexponential} implies that
 	\[
 	\PP\bgg[ \bg|I_w+y+y'+y''-\log x\bg|\le z\bgg]  \le e^{-w(\log x-y-y'-y'')} (e^{wz}-e^{-wz}) \ll wz e^{-w(\log x-y-y'-y'')},
 	\]
	since $w\le(\log_2x)^{-10}$ and $z\le 4\log_2x$. Consequently,
		\[
		\begin{split}
	E(z) &\ll 
	\mathop{\int\cdots\int}\limits_{\substack{0<w<w'<w'' \\ y,y',y''\ge0 \\  y+y'+y''\le \log x+z }} r(y')r(y'') 
		 zwe^{-w\log x-(w'-w)(y'+y'')-(w''-w')y''}\dee w\cdots \dee y'' \\
		&= \frac{z}{(\log x)^2} \iiiint\limits_{\substack{t,y,y',y''\ge 0 \\  y+y'+y''\le \log x+z }} \frac{r(y')r(y'')}{y''(y'+y'')}
		t e^{-t} \dee t\dee y\dee y'\dee y'' ,
		\end{split}
	\]
	where we made the change of variables $t=w\log x$. Since $\int_0^\infty r(u)/u \dee u\ll 1$ and $\log x+z\ll\log x$, relation \eqref{eq:E(z)} follows. This completes the proof of the lemma.
	\end{proof}
	
	\begin{lem}
		\label{lem:prob e2c e3c}
		We have $\PP[\EC_1 \cap \EC_2^c] + \PP[\EC_1 \cap \EC_3^c] \ll 1/\log x$ for $x\ge2$.
	\end{lem}
	
	\begin{proof} Recall that $\EC_2$ failing means that there exists $i\ge1$ such that $W_i/W_0\le 1+\eta/\min\{Y_0, Y_i\}^2$. In addition, recall that the event $\EC_1$ implies that $Y_0 > e^{-\gamma}$ (this was explained in the beginning of the proof of  Lemma \ref{lem:decomposition J neq J*}). Hence, we have that
		\[
			\PP[\EC_1 \cap \EC_2^c] \le 
				\EE\bggg[ \mathop{\sum\sum}_{(W,Y),(W',Y')\in\RS} \one_{1<W'/W\le 1+\eta/\min\{Y,Y'\}^2}\cdot\one_{Y',Y>e^{-\gamma}} \cdot \one_{I_{W, W'} + Y'\in(\log x-Y,\log x]}  \bggg] 
		\]
		with $I_{t, t'} \coloneqq \sum_{(W, Y) \in \RS,\ W \in \RR_{>t}\setminus\{t'\}} Y$. We use Mecke's equation as in the proof of \eqref{eq:E(z)} to get that
		\[
		\PP[\EC_1 \cap \EC_2^c] \le 
		 \iiiint\limits_{\substack{0 < w < w' < w(1+\eta/\min\{y,y'\}^2)\\ y, y' > e^{-\gamma},\ y'\le \log x}} \PP\bg[\log x - y<I_w + y' \le \log x\bg]\cdot e^{-wy-w'y'} 
		 \dee w\dee w'\dee y\dee y' .
		\]
		We have $e^{-w'y'}\le e^{-wy'}$, thus
		\[
		\PP[\EC_1 \cap \EC_2^c] \ll
		\iiint\limits_{\substack{w>0,\ y, y' > e^{-\gamma} \\ y'\le \log x}} 
			\PP\bg[\log x - y<I_w + y' \le \log x\bg]\cdot \frac{we^{-w(y+y')}}{\min\{y,y'\}^2} \dee w\dee y\dee y'.
		\]
By Proposition \ref{prop:icexponential}, we have 
		\[
			\PP\bg[\log x - y<I_w + y' \le \log x\bg] \le \begin{cases}
			e^{-w(\log x - y-y')}&\text{if $y \le \log x$,} \\
				1 & \text{if $y > \log x$.}
			\end{cases}
		\]
	Therefore,
		\begin{align*}
				\PP[\EC_1 \cap \EC_2^c] 
			&\ll \iiint\limits_{\substack{w>0,\ y, y' > e^{-\gamma} \\ y,y'\le \log x}} 
			 \frac{we^{-w\log x}}{\min\{y,y'\}^2}\dee w\dee y\dee y' 
			 +\iiint\limits_{\substack{w>0,\ y, y' > e^{-\gamma} \\ y'\le \log x<y}} 
			\frac{we^{-w(y+y')}}{(y')^2}\dee w\dee y\dee y' \\
			&\le \iint\limits_{e^{-\gamma}<y, y'\le \log x}
			\frac{1}{\min\{y,y'\}^2(\log x)^2}\dee y\dee y' 
			+\iint\limits_{\substack{y, y' > e^{-\gamma} \\ y'\le \log x<y}} 
			\frac{1}{(yy')^2} \dee y\dee y' \\
			&\ll 1/\log x.
		\end{align*}
		This completes the proof of the claimed bound on $\PP[\EC_1\cap \EC_2^c]$.
		
		Finally, we bound $\PP[\EC_1 \cap \EC_3^c]$ using a very similar argument. We have
		\[
		\begin{split}
		\PP[\EC_1 \cap \EC_3^c] &\le 
		\EE\bggg[ \mathop{\sum\sum}_{(W,Y),(W',Y') \in\RS} \one_{1<W/W'\le 1+\eta/\min\{Y,Y'\}^2}\cdot\one_{Y',Y>e^{-\gamma}} \cdot \one_{I_W\in(\log x-Y,\log x]}  \bggg] \\
		&= 
		\iiiint\limits_{\substack{0 < w'<w < w'(1+\eta/\min\{y,y'\}^2)\\ y, y' > e^{-\gamma}}}  \PP\bg[\log x - y<I_w \le \log x\bg]\cdot e^{-wy-w'y'} 
		\dee w\dee w'\dee y\dee y' \\
		&\le J_1+J_2,
		\end{split}
		\]
		where
		\[
		J_1 \coloneqq \iiiint\limits_{\substack{0 < w'<w < w'(1+\eta/\min\{y,y'\}^2)\\ y, y' > e^{-\gamma},\ y\le \log x}} e^{-w\log x-w'y'}
		\dee w\dee w'\dee y\dee y' 
		\]
		and
		\[
		J_2 \coloneqq \iiiint\limits_{\substack{0 < w'<w < w'(1+\eta/\min\{y,y'\}^2)\\ y>\log x,\ y' > e^{-\gamma}}}  e^{-wy-w'y'} 
		\dee w\dee w'\dee y\dee y' .
		\]
		Using $e^{-w\log x}\le e^{-w'\log x}$, we find that
		\begin{align*}
			J_1&\ll \iiint\limits_{\substack{w'>0,\ y, y' > e^{-\gamma}\\ y\le \log x}} \frac{w'e^{-w'(y'+\log x)}}{\min\{y,y'\}^2}  \dee w'\dee y\dee y' \\
			&= \iint\limits_{\substack{y, y' > e^{-\gamma} \\ y\le \log x}} \frac{1}{\min\{y,y'\}^2(y'+\log x)^2} \dee y\dee y'\\
			& \ll1/\log x. 
		\end{align*}
		Similarly, we have	
		\begin{align*}
			J_2&\ll \iiint\limits_{\substack{w'>0,\ y' > e^{-\gamma}\\ y>\log x}} \frac{w'e^{-w'(y+y')} }{\min\{y,y'\}^2}  \dee w'\dee y\dee y' \\
			&= \iint\limits_{\substack{y' > e^{-\gamma} \\ y>\log x}} \frac{1}{\min\{y,y'\}^2(y+y')^2} \dee y\dee y'\\
			& \ll1/\log x. 
		\end{align*}
		This implies that $\PP[\EC_1\cap\EC_3^c]\ll 1/\log x$, thus completing the proof of the lemma. 
	\end{proof}

	As an immediate corollary of Lemmas \ref{lem:decomposition J neq J*}, \ref{lem:prob e1c} and \ref{lem:prob e2c e3c}, we have:
	
	\begin{prop}
		\label{prop:prob J neq J*}
		For $x\ge2$, we have $\PP\bg[J_x \ne J_x^*\bg] \ll 1/\log x$.
	\end{prop}

	
	\section{Proof of Proposition \ref{prop:dtv of M and N}}
	\label{sec:dtv}
	
	With a good estimation of the distribution of $J_x^*$ and with the fact that $J_x = J_x^*$ with a quantifiably high probability, we are able to give an upper bound on the total variation distance between the distribution of $M_x$ and $N_x$.
	
	\begin{proof}[Proof of Proposition \ref{prop:dtv of M and N}]
		Recall that we constructed $N_x$ in the coupling such that $\PP[M_x \ne N_x] = \dtv(\mu_x, \nu_x)$ with $\mu_x$ and $\nu_x$ being the distribution of $M_x$ and $N_x$ respectively. Note that $M_x > x$ must imply that $J_x > x$, which means that $\EC_1$ cannot happen in that situation by \eqref{eq:e1 implies J < x}. Hence, using Lemma \ref{lem:prob e1c}, we see that it is enough to show that 
		\[
			\PP[M_x \in A] = \frac{\# A}{\floor x} + O\bgg(\frac{1}{\log x}\bgg)
		\]
		uniformly over all $A \subseteq \ZZ \cap [1, x]$.
		
		For each $j \le x$, we define the sets
		\[
			A_j \coloneqq \bigcup_{\substack{p:\ pj \in A}} \left(\frac{\theta(p-1)}{\theta(x/j)}, \frac{\theta(p)}{\theta(x/j)}\right].
		\]
		Note that $M_x \in A$ and $J_x \le x/2$ if, and only if, $J_x=j$ for some $j\le x/2$ and $U'_1 \in A_j$. With this in mind, we define the following events:
		\begin{itemize}
			\item $\mathcal B_1 \coloneqq \{J_x=J_x^* \le x/2\}$;
			\item $\mathcal B_2 \coloneqq \{ M_x \in A\}$;
			\item $\mathcal B_3 \coloneqq \{\mbox{$J_x^* = j$ and $U'_1 \in A_j$ for some $j\le x/2$}\}$.
		\end{itemize} 
		Since $\mathcal B_1 \cap \mathcal B_2 = \mathcal B_1 \cap \mathcal B_3$, we have
		\[
			\bgg|\PP[\BC_2] - \PP[\BC_3]\bgg| \le \PP[\BC_1^c] \le \PP\bg[J_x \ne J_x^*\bg] + \PP\bg[x/2< J_x^* \le x\bg] \ll \frac{1}{\log x}
		\]
		by Lemma \ref{lem:probaj} and by Proposition \ref{prop:prob J neq J*}. Therefore, we have 
		\begin{align}
			\PP\bg[M_x \in A\bg] 
				&= \sum_{j \le x/2} \PP\bg[J_x^* = j,\ U'_1 \in A_j\bg] + \bigoo{\frac{1}{\log x}} \nonumber  \\
				&= \mathop{\sum\sum}_{p,j:\ pj \in A} \frac{\log p}{\theta(x/j)}\cdot \PP[J_x^* = j] + \bigoo{\frac{1}{\log x}} .
			\label{eq:Mx reduction 1}				
		\end{align}
Since $\theta(t)/t = 1 + O\bg((\log t)^{-2}\bg)$ for all $t \ge 2$ by the Prime Number Theorem \cite[Theorem 8.1]{Koukoulo19}, we have
		\begin{equation}
			\label{eq:Mx reduction 2}
			\frac{\PP[J_x^* = j]}{\theta(x/j)} = \frac{1}{\floor x \log x}\left(1+\bigoo{\frac{1}{(\log(x/j))^2}+\frac{1}{\log x}}\right)
		\end{equation}
		for $j\le x/2$. In addition, note that
		\begin{equation}
			\label{eq:Mx reduction 3}
			\mathop{\sum\sum}_{p,j:\ pj \in A}  \log p \cdot \left(\frac{1}{(\log(x/j))^2} + \frac{1}{\log x}\right) \ll \sum_{j \le x/2} \left(\frac{x}{j(\log(x/j))^2} + \frac{x}{j\log x}\right) \ll x.
		\end{equation}
		By combining \eqref{eq:Mx reduction 1}, \eqref{eq:Mx reduction 2} and \eqref{eq:Mx reduction 3}, we reduce the problem to showing the following:
		\begin{equation}
			\label{eq:Mx reduction 4}
			 \mathop{\sum\sum}_{p,j:\ pj \in A}  \frac{\log p}{\log x} = \# A + O\bgg(\frac{x}{\log x}\bgg).
		\end{equation}
		If we set $L(a)\coloneqq \sum_{p|a}\log p$, then we have
		\begin{equation}
			\label{eq:Mx reduction 5}
		 \mathop{\sum\sum}_{p,j:\ pj \in A}  \frac{\log p}{\log x} = \sum_{a\in A}\frac{L(a)}{\log x} = \# A 
		 	- \sum_{a\in A} \frac{\log(x/a) + (\log a - L(a))}{\log x}
		\end{equation}
		The quantity $\log(x/a)$ is non-negative whenever $a \le x$. Furthermore, the integer $a/(\prod_{p|a} p)$ is always a divisor of $s(a)$, thus $0 \le \log a - L(a) \le \log s(a)$. Therefore, 
		\begin{align*}
			\sum_{a\in A} \frac{\log(x/a) + (\log a - L(a))}{\log x} &\ll \sum_{a\le x} \frac{\log(x/a) + \log s(a)}{\log x} \\
			&= \frac{\floor x}{\log x} \cdot \EE\bg[\log(x/N_x) + \log s(N_x)\bg] \\
			&\ll \frac{x}{\log x}
		\end{align*}
		by Lemma \ref{lem:properties of N}. With this inequality combined with \eqref{eq:Mx reduction 5}, we establish \eqref{eq:Mx reduction 4}, and hence it completes the proof of Proposition \ref{prop:dtv of M and N}.
	\end{proof}

	\newpage
	
	\clearpage
	\thispagestyle{fancy}
	\fancyhf{} 
	\renewcommand{\headrulewidth}{0cm}
	\lhead[{\scriptsize \thepage}]{}
	\rhead[]{{\scriptsize\thepage}}
	\part{Factorization into $k$ parts}\label{part:DDT}

	\section{Theorem \ref{thm:fact into k parts} for the probabilistic model}\label{sec:Donnelly-Tavare}
	
	With the following probabilistic version of Theorem \ref{thm:fact into k parts} given by Donnelly and Tavaré in 1987 \cite[Section 3]{DonnellyTavare87}, we see how a Poisson--Dirichlet process relates to the Dirichlet distribution. We repeat their proof since they use a Poisson process that will come in handy later.

	\begin{prop}[Donnelly--Tavar\'e \cite{DonnellyTavare87}]
		\label{prop:probabilistic dirichlet law}
		Let $\mathbf V = (V_1, V_2, \ldots)$ be a Poisson--Dirichlet process, let $\boldalpha = (\alpha_1, \ldots, \alpha_k) \in \simplex$ such that $\alpha_i > 0$ for all $1 \le i \le k$, and let $(C_i)_{i \ge 1}$ be a sequence of i.i.d.~random variables (also independent of $\mathbf V$) with $\PP[C_i = j] = \alpha_j$ for all $i \ge 1$. Then the random vector 
		\[
		\left(\sum_{i:\ C_i = 1} V_i, \ldots, \sum_{i:\ C_i = k} V_i\right)
		\]
		is distributed according to $\Dir(\boldalpha)$.
	\end{prop}
	
	\begin{proof}
		Let $X_1 > X_2 > \cdots$ be the points of a Poisson process on $\RR_{>0}$ with intensity $\frac{e^{-x}}{x} \, \mathrm dx$, and let $S \coloneqq \sum_{i \ge 1} X_i$. For each $i \ge 1$, Let $V_i \coloneqq \frac{X_i}{S}$. We then know that $(V_1, V_2, \ldots)$ follows the Poisson--Dirichlet distribution \cite[Theorem 2.2]{Feng10}. With the Colouring Theorem (Proposition \ref{prel:colour}), the point processes
		\[
		\Pi_m = \{X_i : C_i = m\}
		\]
		form independent Poisson processes of intensity $\alpha_m\cdot \frac{e^{-x}}{x} \, \mathrm dx$ for all $m = 1, \ldots, k$. We then let
		\[
		S_m \coloneqq \sum_{X \in \Pi_m} X
		\]
		for $m=1,2,\dots,k$. 	With Campbell's Theorem (Proposition \ref{prel:campbell}), we compute the moment generating function 
		\[
		\log \EE[e^{sS_m}] = \alpha_m\cdot \riemint{0}{\infty}{\frac{e^{sx}-1}{x}\cdot e^{-x}}{x} = -\alpha_m\log(1-s)
		\] 
		for $\Re(s) < 1$. This coincides with the moment generating function of $\text{Gamma}(\alpha_m, 1)$ distribution. We thus deduce that the vector
		\[
		\left(\sum_{i:\ C_i = 1} V_i, \ldots, \sum_{i:\ C_i = k} V_i\right) = \left(\frac{S_1}{\sum_{m=1}^k S_m}, \ldots, \frac{S_k}{\sum_{m=1}^k S_m}\right)
		\] 
		follows the distribution $\text{Dir}(\boldalpha)$ (see \cite[Chapter 49, pp. 485-487]{KotzBal00} for a proof). 
	\end{proof}

	\section{A coupling for $\mathbf{D}_{f, x}$}\label{sec:coupling for factorizations}
	
	Fix $k\ge2$, $\boldalpha \in \simplex$ with $\alpha_i>0$ for all $i$, and $f\in\mathcal{F}_k(\boldalpha)$. In addition, let $N_x$ be a random integer uniformly distributed in $\ZZ \cap [1, x]$, let $\mathbf D_{f, x}$ be a random vector satisfying \eqref{eq:prob mass function f}, and let $\mathbf Z$ be a $\simplex$-valued random variable distributed according to $\text{Dir}(\boldalpha)$. We also consider the random vector $\bolddelta_{f, x} \in \Delta^{k-1}$ defined by
	\begin{equation}
		\label{eq:defn delta}
		\bolddelta_{f, x} \coloneqq  \bggg(\frac{\log D_{f,x,1}}{\log N_x},\dots,\frac{\log D_{f,x,k}}{\log N_x}\bggg)
	\end{equation}
	when $N_x \ge 2$, and $\bolddelta_{f,x} = (1,0,\ldots,0)$ when $N_x = 1$. We note that Theorem \ref{thm:fact into k parts} is about comparing the cumulative distribution function of $\bolddelta_{f, x}$ with the one for $\mathbf Z$. To achieve this, we construct an appropriate coupling between $N_x$, $\mathbf D_{f, x}$ and $\mathbf Z$. This will build on the coupling given in Section \ref{sec:coupling}.
	
	\begin{lem}
		\label{lem:coupling for factorizations}
			Let $x\ge2$ and assume the above notation. There exists a coupling of $N_x$, $\mathbf V$, $\mathbf D_{f, x}$ and $\mathbf Z$ such that $\PP[\mathcal E^c] \ll \frac{1}{\log x}$ with
			\[
				\mathcal E \coloneqq \bgggg\{\maxnorm{\bolddelta_{f, x} - \mathbf Z} \le \frac{2 \cdot \log(x/N_x) + 3 \cdot \log s(N_x) + 2\cdot \Theta_x}{\log x}\bgggg\},
			\]
			with $\bolddelta_{f, x}$ defined as in \eqref{eq:defn delta}, $\Theta_x$ defined as in \eqref{eq:def of Theta_x}, and $\maxnorm \cdot$ being the supremum norm in $\RR^k$.
	\end{lem}
	
	\begin{proof} Our starting point is the coupling of $N_x$ and $\mathbf V$ described in Section \ref{sec:coupling}, and it remains to construct $\mathbf D_{f, x}$ and $\mathbf Z$. To do this, we need to equip this space with the following additional random variables that are all independent of each other and of $N_x$ and $\mathbf V$, and whose role will become apparent later:
		\begin{itemize}
			\item a sequence $(C_i')_{i\ge1}$ of i.i.d. random variables such that $\PP[C_i'=\ell]=\alpha_\ell$ for all $i\ge1$ and all $\ell=1,2,\dots,k$;
			\item for each natural number $n$ and each prime $p$, let $\mathbf X_{p}(n) \coloneqq \bg(X_{p, 1}(n), \ldots, X_{p, k}(n)\bg)$ be independent random vectors with
			\[
			\PP\bg[X_{p, i}(n) = e_i \ \forall i \le k\bg] = f(p^{e_1}, \ldots, p^{e_k}),
			\]
			for every non-negative integer solutions to $e_1+\cdots + e_k = \nu_p(n)$. 
		\end{itemize}
	To construct $\mathbf D_{f, x}$, we may take
	\[
	D_{f,x,i} \coloneqq \prod_{p} p^{X_{p, i}(N_x)}\quad\text{for}\ i=1,2,\dots,k.
	\]
	Since $f \in \mathcal F_k(\boldalpha)$ satisfies property (c) of Definition \ref{dfn:Fktheta}, we may easily check that the distribution of $\mathbf D_{f, x}=(D_{f,x,1},\dots,D_{f,x,k})$ is indeed in accordance with \eqref{eq:prob mass function f}.
	
	Next, we construct $\mathbf Z$. Firstly, we need to construct a new sequence of random variables $(C_i)_{i\ge1}$, whose definition depends deterministically on the previously introduced random variables. The following events must occur:
	\begin{itemize}
		\item We have $N_x = n$ for some choice of $n \in \ZZ \cap [1, x]$;
		\item For all $j \in \{1, \ldots, k\}$ and all $p|n^\flat$, we have $X_{p, j}(n) = \one_{j=\ell(p)}$ for some choice of sequence $(\ell(p))_{p | n^\flat}$ taking values in $\{1, \ldots, k\}$.
		\item For all $i \ge 1$, we have that $C'_i = \ell'_i$ for some choice sequence $(\ell'_i)_{i \ge 1}$ taking values in $\{1, \ldots, k\}$.
	\end{itemize}
	Recall the factorization $N_x=P_1P_2\cdots$ with the $P_i$'s forming a non-increasing sequence of primes or ones. Then we construct $(C_i)_{i \ge 1}$ in the following way:
	\begin{itemize}
		\item If $i\ge1$ is such that $P_i > 1$ and $P_i|n^\flat$,  then we let $C_i\coloneqq\ell(P_i)$.
		\item Otherwise, we let $C_i\coloneqq \ell_i'$. 
	\end{itemize}
	Since $f\in\mathcal F_k(\boldalpha)$ satisfies property (b) of Definition \ref{dfn:Fktheta}, a straightforward computation reveals that for any $n \in \ZZ \cap [1, x]$, any $(\ell_i)_{i \in I}$ taking values in $\{1, \ldots, k\}$ and any finite set $I \subset \NN$, we have 
	\[
		\PP\bg[C_i = \ell_i \ \ \forall i \in I\, |\,N_x = n\bg] = \prod_{j=1}^k \alpha_j^{\#\{i \in I : \ell_i = j\}},
	\]
	implying that the random variables $(C_i)_{i\ge1}$ are independent of each other, of $N_x$ and of $\mathbf V$, and they satisfy 
	\[
	\label{eq:C_i dfn}
	\PP[C_i=\ell]=\alpha_\ell\quad\text{for}\ i=1,2,\dots\ \text{and}\ \ell=1,2,\dots,k.
	\]
	Finally, equipped with these random variables and motivated by the proof of Donnelly and Tavar\'e in Section \ref{sec:Donnelly-Tavare}, we take
	\[
	\mathbf Z \coloneqq \left(\sum_{i:\, C_i=1} V_1, \ldots, \sum_{i:\, C_i=k} V_k  \right),
	\]
	which is distributed according to $\text{Dir}(\boldalpha)$ by the discussion in Section \ref{sec:Donnelly-Tavare}.
	
	Recall the definition of $\bolddelta_{f, x}$ in \eqref{eq:defn delta}. We introduce the auxiliary random variables
	\begin{equation}
		\label{eq:defn delta* and rho}
		\bolddelta_{f, x}^* \coloneqq \left(\frac{\log D_{f,x,1}}{\log x}, \ldots, \frac{\log D_{f,x,k}}{\log x}\right),
		\qquad \boldrho_{f, x} \coloneqq \left(\sum_{i:\, C_i=1} \frac{\log P_i}{\log x} , \ldots, \sum_{i:\, C_i=k} \frac{\log P_i}{\log x}  \right).
	\end{equation}
	We claim that
	\begin{equation}
		\label{eq:transition delta to delta*}
		\maxnorm{  \bolddelta_{f, x} -  \bolddelta_{f, x}^*} \le \frac{\log(x/N_x)}{\log x}.
	\end{equation}
	Indeed, if $N_x = 1$, both the left-hand and right-hand sides of the inequality are always equal to 1. For $N_x \ge 2$, we note that
	\[
		\maxnorm{  \bolddelta_{f, x} -  \bolddelta_{f, x}^*} = \frac{\log(x/N_x)}{\log x} \cdot  \max_{1 \le i \le k} \frac{\log D_{f,x,i}}{\log N_x},
	\]
	from which \eqref{eq:transition delta to delta*} follows.
	
	Furthermore, we also assert that
	\begin{equation}
		\label{eq:transition rho to delta*}
		\maxnorm{\bolddelta_{f, x}^* - \boldrho_{f, x}} \le \frac{\log s(N_x)}{\log x}.
	\end{equation}
	To verify this, suppose $N_x=n$. Then we have
	\begin{align*}
		\log D_{x,\ell} &= \sum_{p|n^\flat}  X_{p, \ell}(n) \log p +  \sum_{p|s(n)}  X_{p, \ell}(n) \log p \\
		&=\sum_{i\ge1:\ P_i|n^\flat} X_{P_i, \ell}(n) \log P_i +  \sum_{p|s(n)}  X_{p,\ell}(n) \log p \\
		&=\sum_{i\ge1:\ P_i|n^\flat}  \one_{C_i=\ell} \log P_i +  \sum_{p|s(n)}  X_{p,\ell}(n) \log p ,
	\end{align*}
	whence \eqref{eq:transition rho to delta*} follows readily. 
	
	Now if the event $\mathcal E$ does not happen, then
	\[
	\maxnorm{\bolddelta_{f, x} - \mathbf Z} > \frac{2 \cdot \log(x/N_x) + 3 \cdot \log s(N_x) + 2\cdot \Theta_x}{\log x}.
	\]
	Together with	 \eqref{eq:transition delta to delta*} and \eqref{eq:transition rho to delta*}, this implies that
	\[
		\maxnorm{\boldrho_{f, x} - \mathbf Z} > \frac{\log(x/N_x) + 2 \cdot \log s(N_x) + 2\cdot \Theta_x}{\log x}
	\]
	By Lemma \ref{lem:coupling ineq}, we must then have that $M_x \ne N_x$. In conclusion, we have proven that $\mathcal{E}^c\subseteq\{M_x\ne N_x\}$. Thus, Proposition \ref{prop:dtv of M and N} shows that $\PP[\mathcal E^c] \ll \frac{1}{\log x}$, as required.
	\end{proof}

	In order to make use of Lemma \ref{lem:coupling for factorizations}, we need to introduce some further notation. Let
	\[
	\Rnt \coloneqq 3\cdot \frac{\log(x/N_x^\flat)}{\log x} \quad \text{and} \quad \Rprob \coloneqq 2\cdot \frac{\Theta_x}{\log x}.
	\]
	We write $B(\mathbf x,r)$ for the closed ball of radius $r \ge 0$ centered at $\mathbf x$ in the normed vector space $(\RR^k, \maxnorm{\cdot})$. We recall that $\simplex$ is the standard $(k-1)$-dimensional simplex. For any subset $A \subseteq \simplex$, we write $\partial A$ for the boundary of $A$ in the relative topology of $\simplex$.
	
	In addition, we define the following events:
	\begin{itemize}
		\item $\ECnt(A) \coloneqq \bg\{ B(\bolddelta_{f, x}^*, \Rnt)\cap \partial A=\emptyset\bg\}$;
		
		\item $\ECprob(A) \coloneqq \bg\{  B(\mathbf Z, \Rprob) \cap \partial A = \emptyset \bg\}$.
	\end{itemize}

	\begin{lem}
		\label{lem:sym diff D and Z}
		For any measurable set $A\subseteq\simplex$, we have 
		\[
		\PP\bg[\bolddelta_{f, x} \in A\bg] = \PP\bg[\mathbf Z \in A\bg] + O\left(\PP\bg[\ECnt(A)^c\bg] + \PP\bg[\ECprob(A)^c\bg] + \frac{1}{\log x}\right).
		\]
	\end{lem}
	
	\begin{proof} Let $\EC$ be as in Lemma \ref{lem:coupling for factorizations}. Hence, it suffices to show that
		\begin{equation}
			\label{eq:sym diff D and Z on E}
				\bgg|	\PP\bg[\EC\cap\{\bolddelta_{f, x} \in A\}\bg] - \PP\bg[\EC\cap\{\mathbf Z \in A\}\bg]\bgg| \le \PP\bg[\ECnt(A)^c\bg] + \PP\bg[\ECprob(A)^c\bg] .
		\end{equation}
		We shall prove the following stronger statement: if $\EC\cap \ECnt(A)\cap \ECprob(A)$ occurs, then $\bolddelta_{f, x}\in A$ if, and only if, $\mathbf Z\in A$. 
		
		Indeed, assume that $\EC\cap \ECnt(A)\cap \ECprob(A)$ occurs, and we also have $\bolddelta_{f, x}\in A$. Aiming for a contradiction, we also assume that $\mathbf Z\notin A$. 
		
		Let $\Rnt' \coloneqq (\log x)^{-1}(2\cdot \log(x/N_x) + 3\cdot \log s(N_x))$. Note that $B(\bolddelta_{f, x}, \Rnt') \subseteq B(\bolddelta_{f, x}^*, \Rnt)$ by \eqref{eq:transition delta to delta*}. By assumption of $\ECnt(A)$, we know that $B(\bolddelta_{f, x},\Rnt')\cap\partial A=\emptyset$. Since $\bolddelta_{f, x} \in A$ and $B(\bolddelta_{f, x},\Rnt')\cap \simplex$ is connected (being the intersection of two convex sets), we have that
		\[
		B(\bolddelta_{f, x},\Rnt')\cap \simplex \subseteq A. 
		\]
		And by assumption of $\ECprob(A)$, we have $\mathbf Z \in \simplex \setminus A$ and $B(\mathbf Z,\Rprob)\cap\partial A=\emptyset$, which implies
		\[
		B(\mathbf Z,\Rprob)\cap \simplex \subseteq \simplex \setminus A. 
		\]
		In particular, we find that $B(\bolddelta_{f, x}, \Rnt')\cap B(\mathbf Z,\Rprob)\cap \simplex =\emptyset$. On the other hand, we know that $\maxnorm{\bolddelta_{f, x}-\mathbf Z}\le \Rnt' +\Rprob$ when $\EC$ occurs by Lemma \ref{lem:coupling for factorizations}(b). In particular, there exists a point $\mathbf y$ on the line segment connecting $\bolddelta_{f, x}$ and $\mathbf Z$ such that $\maxnorm{\bolddelta_{f, x}-\mathbf y}\le \Rnt'$ and $\maxnorm{\mathbf y-\mathbf Z}\le \Rprob$. Since the sets $\simplex$,  $B(\bolddelta_{f, x},\Rnt')$ and $B(\mathbf Z,\Rprob)$ are convex, we conclude that $\mathbf y$ lies in their intersection. But we had seen before that this intersection is the empty set. We have thus arrived at a contradiction. This completes the proof that if $\EC\cap \ECnt(A)\cap \ECprob(A)$ occurs, and we also know that $\bolddelta_{f, x}\in A$, then we must also have that $\mathbf Z\in A$.
		
		Conversely, we may show by a simple variation of the above argument that if $\EC\cap \ECnt(A)\cap \ECprob(A)$ occurs, and we also know that $\mathbf Z\in A$, then we must also have that $\bolddelta_{f, x}\in A$. This completes the proof of the lemma.
	\end{proof}
	
	To prove Theorem \ref{thm:fact into k parts}, we need to bound $\PP[\ECnt(A)^c]$ and $\PP[\ECprob(A)^c]$ when $A$ equals the set
	\[
	\simplex_{\mathbf{u}} \coloneqq \{\mathbf x\in \simplex: x_i\le u_i\ \forall i<k\} .
	\]
with $\mathbf u\in(0,1]^{k-1}$. Note that
\[
\partial \bg(\simplex_{\mathbf{u}}\bg) \subseteq \{\mathbf x \in \simplex : \exists i < k \text{ such that } x_i = u_i < 1\},
\]
(here it is important that the boundary of $\simplex_{\mathbf{u}}$ is defined with respect to the topology of $\simplex$). Therefore, 
\[
\ECnt(\simplex_{\mathbf{u}})^c\subseteq \bigcup_{\substack{i<k \\ u_i\neq 1}} \Bnt^{(i)}(u_i),
	\quad\text{where}\quad
	\Bnt^{(i)}(u) \coloneqq \bgg\{ \bg| \log D_{f,x,i} - u\log x\bg|\le 3\log(x/N_x^\flat)	\bgg\},
\]
as well as
\[
\ECprob(\simplex_{\mathbf{u}})^c\subseteq \bigcup_{\substack{i< k \\ u_i\neq 1}} \Bprob^{(i)}(u_i),
\quad\text{where}\quad
\Bprob^{(i)}(u) \coloneqq \bggg\{ \bg| Z_i - u \bg|\le \frac{2\Theta_x}{\log x}  \bggg\},
\]
We then have the following two crucial estimates:

	\begin{lem}
		\label{lem:bound of prob of ECnt}
	For $i\in\{1,2,\dots,k-1\}$ and $u\in(0,1)$, we have the uniform estimate
	\[
	\PP\bg[\Bnt^{(i)}(u) \bg] \ll  \frac{1}{(1+ u\log x)^{1-\alpha_i}(1+(1-u)\log x)^{\alpha_i}}  .
	\]
	\end{lem}
	
	\begin{lem}
	\label{lem:bound of prob of ECprob}
		For $i\in\{1,2,\dots,k-1\}$ and $u\in(0,1)$, we have that
		\[
		\PP\bg[\Bprob^{(i)}(u) \bg] \ll  \frac{1}{(1+ u\log x)^{1-\alpha_i}(1+(1-u)\log x)^{\alpha_i}}  .
		\]
	\end{lem}
	
Using Lemmas \ref{lem:sym diff D and Z}-\ref{lem:bound of prob of ECprob} together with Proposition \ref{prop:probabilistic dirichlet law} yields immediately Theorem \ref{thm:fact into k parts}. Thus, it remains to prove Lemmas \ref{lem:bound of prob of ECnt} and \ref{lem:bound of prob of ECprob}, which we do in the next two sections.

%
%
%
%
%
%
%

	\begin{remss}(a) 
		The proofs of Lemmas \ref{lem:bound of prob of ECnt} and \ref{lem:bound of prob of ECprob} are rather involved. However, it is possible to obtain slightly weaker versions of them in a rather easy manner, which will then lead to a correspondingly weaker version of Theorem \ref{thm:fact into k parts}. 
		
		First, we prove a variation of Lemma \ref{lem:sym diff D and Z}. Let $\EC^*(A)$ be the event that $B(\mathbf Z,R^*)\cap \partial A=\emptyset$ with 
		\[
		R^* \coloneqq \frac{2\cdot \log(x/N_x) + 3\cdot \log s(N_x) + 2\cdot \Theta_x}{\log x}.
		\] 
	 	A straightforward modification of the proof of Lemma \ref{lem:sym diff D and Z} implies that
		\begin{equation}
			\label{eq:symm diff for weaker thm}
		\PP[\mathbf D_{f, x} \in A] = \PP[\mathbf Z \in A] +  O\left(\PP\bg[\EC^*(A)^c\bg]  + \frac{1}{\log x}\right) 
				\end{equation}
		for any measurable set $A$. With Chernoff's bound and Lemmas \ref{lem:properties of N}-\ref{lem:properties of Theta}, the probability that $R^*>\frac{\log_2 x}{\log x}$ is $\ll 1/\log x$. Moreover, for any $\delta\in(0,1/4]$, we can show by a direct computation with the Dirichlet distribution the uniform bound
		\[
		\PP\bgg[\exists j < k\ \text{such that}\ |Z_j-u_j|\le \delta \bgg] \ll \sum_{1\le j<k} \frac{\delta}{(u_j+\delta)^{1-\alpha_j}(1-u_j+\delta)^{\alpha_j}}
		\]
		for all $\mathbf u\in [0,1]^{k-1}$ (see Lemma \ref{lem:beta estimate} below for a proof of this claim). Taking $\delta=\frac{\log_2x}{\log x}$ proves that
				\[
		\PP\bg[D_{f,x,i} \le N_x^{u_i}\ \, \forall i < k\bg] = F_{\boldalpha}(\mathbf u)
		+ O\bgggg(\sum_{\substack{1 \le i < k \\ u_i \ne 1}} \frac{1}{(1+u_i\frac{\log x}{\log_2x})^{1-\alpha_i}(1+(1-u_i)\frac{\log x}{\log_2x})^{\alpha_i}}\bgggg).
		\]
		
		\medskip
		
		(b) As a matter of fact, the above proof could have worked with Arratia's original coupling from \cite{Arratia02}. Under this coupling, we have $\PP[M_x \ne N_x] \ll \frac{\log_2x}{\log x}$ (see Remark (b) at the end of Section \ref{sec:coupling}), so that Lemma \ref{lem:coupling for factorizations} would hold with part (a) replaced by the weaker bound $\PP[\EC^c]\ll \frac{\log_2x}{\log x}$, and thus \eqref{eq:symm diff for weaker thm} would hold with $\frac{\log_2x}{\log x}$ in place of $\frac{1}{\log x}$. 
	\end{remss}

	\section{Proof of Lemma \ref{lem:bound of prob of ECnt} using elementary number-theoretic techniques}\label{sec:boundary event NT}
	
	Note that 
	\begin{equation}
		\label{eq:small square-free part}
			\PP\bg[N_x^\flat\le x/(\log x)^3\bg] \le \frac{1}{\log x} \EE\bg[(x/N_x^\flat)^{1/3}\bg] \ll \frac{1}{\log x}
	\end{equation}
	by Lemma \ref{lem:properties of N} applied with $\alpha=\beta=1/3$. 
	 
	Fix now $i\in\{1,\dots,k-1\}$ and $u\in(0,1)$. For simplicity, let us write 
	\[
	\alpha\coloneqq \alpha_i
	\]
	for the remainder of this section. Given $z$, let us define
	\[
	S(z)\coloneqq \sum_{\substack{n\le x \\ n^\flat \in(\frac{x}{z},\frac{2x}{z}]}} \sum_{\substack{d_1\cdots d_k = n \\ d_i\in[x^uz^{-3}, x^uz^3] }} f(d_1,\dots,d_k). 
	\]
	Using \eqref{eq:small square-free part} and \eqref{eq:prob mass function f}, we find that
	\begin{equation}
		\label{eq:bounding Bnt 1}
		\PP\bg[\Bnt^{(i)}(u)\bg] \le \frac{1}{\floor{x}} \sum_{1\le m\le 5\log_2x} S(2^m) + O\bggg(\frac{1}{\log x}\bggg) .
	\end{equation}
	Hence Lemma \ref{lem:bound of prob of ECnt} follows readily from the estimate below:
	
	\begin{lem}
		There exists a universal constant $x_0\ge2$ such that if $x\ge x_0$ and $z\in[2,(\log x)^4]$, then we have the uniform estimate
		\[
		S(z) \ll \frac{\log z}{\sqrt{z}} \cdot \frac{x}{(1+ u\log x)^{1-\alpha}(1+(1-u)\log x)^{\alpha}}  .
		\]
	\end{lem}
	
	\begin{proof}
	If $d_1\cdots d_k=n$, then we may uniquely write $d_j=d_j'd_j''$ with $d_j'|n^\flat$ and $d_j''|s(n)$. In particular, $d_j''\le s(n)\le x/n^\flat\le z$ for all $j$. Hence, if $d_i\in[x^u z^{-3},x^uz^3]$, we must also have that $d_i'\in[x^uz^{-4},x^uz^3]$. Hence, using property (c) of Definition \ref{dfn:Fktheta}, we deduce that
	\[
	\sum_{\substack{d_1\cdots d_k = n \\ d_i\in[x^uz^{-3}, x^uz^3] }} f(d_1,\dots,d_k) 
		\le \sum_{\substack{d_1'\cdots d_k' = n^\flat \\ d_i'\in[x^uz^{-4}, x^uz^3] }} f(d_1',\dots,d_k') 
			\sum_{d_1''\cdots d_k'' = s(n)} f(d_1'',\dots,d_k'') .
	\]
	Relation \eqref{eq:f on squarefrees} implies that $f(d_1',\dots,d_k') = \prod_{j=1}^k \alpha_j^{\omega(d_j')}$. Moreover, using property (a) of  Definition \ref{dfn:Fktheta}, we find that the sum over the $d_j''$'s equals $1$. In conclusion,
	\[
	\sum_{\substack{d_1\cdots d_k = n \\ d_i\in[x^uz^{-3}, x^uz^3] }} f(d_1,\dots,d_k) 
	\le \sum_{\substack{d_1'\cdots d_k' = n^\flat \\ d_i'\in[x^uz^{-4}, x^uz^3] }} \prod_{j=1}^k \alpha_j^{\omega(d_j')}
= 	\sum_{\substack{d| n^\flat \\ d\in[x^uz^{-4}, x^uz^3] }}  \alpha^{\omega(d)} (1-\alpha)^{\omega(n^\flat/d)} . 
	\]
	Let $n^\flat=dm$ and $s(n)=b$. If $n^\flat\in[\frac{x}{z},\frac{2x}{z}]$ and $d\in[x^uz^{-4},x^uz^3]$, then $m\in [x^{1-u}z^{-4},2x^{1-u}z^3]$. Hence, we find that
	\begin{align}
	S(z) 	&\le \mathop{\sum\sum}_{\substack{dm \in[\frac{x}{z},\frac{2x}{z}] \\ d\in[x^uz^{-4}, x^uz^3] \\ m\in[x^{1-u}z^{-4},2x^{1-u}z^3] }} 
	 \alpha^{\omega(d)} (1-\alpha)^{\omega(m)} \sum_{\substack{b\le z \\ b\ \text{square-full}}} 1  \nn
	&\ll \sqrt{z} \mathop{\sum\sum}_{\substack{dm \in[\frac{x}{z},\frac{2x}{z}] \\ d\in[x^uz^{-4}, x^uz^3] \\ m\in[x^{1-u}z^{-4},2x^{1-u}z^3] }}  \alpha^{\omega(d)} (1-\alpha)^{\omega(m)} 
	= \sqrt{z}\cdot (S_1+S_2),
	\label{eq:bound in terms of S1 and S2}
	\end{align}
	where $S_1$ denotes the double sum over $d$ and $m$ with the additional constraint $d\le \sqrt{2x/z}$, and $S_2$ is the corresponding sum over pairs $(d,m)$ with $d>\sqrt{2x/z}$ and $m\le \sqrt{2x/z}$. 
	
	First, we bound $S_1$. Note that for the conditions on $d$ to be compatible we must have that $u\le 2/3$, provided that $x$ is large enough. Assuming this is the case, we apply twice Proposition \ref{prel:partial summation bounds} to find that
	\begin{align}
	S_1 &\le \sum_{\substack{d\le \sqrt{2x/z} \\  d\in[x^uz^{-4}, x^uz^3] }}  \alpha^{\omega(d)} \sum_{m\le \frac{2x}{dz}} (1-\alpha)^{\omega(m)}   \nn
	&\ll 	 \sum_{\substack{d\le \sqrt{2x/z} \\  d\in[x^uz^{-4}, x^uz^3] }}  \alpha^{\omega(d)} \cdot \frac{x(\log x)^{-\alpha}}{dz} \nn
	&\ll \frac{x(\log z)(1+ u\log x)^{\alpha-1}(\log x)^{-\alpha}}{z} .
	\label{eq:bound for S1}
	\end{align}
	
	We bound $S_2$ in a similar manner. For this sum to be non-empty, we have $u\ge1/3$. If this is the case, then we have
	\begin{align}
		S_2 
		&\le  \sum_{\substack{m\le \sqrt{2x/z} \\  m\in[x^{1-u}z^{-4}, 2x^{1-u}z^3] }}  (1-\alpha)^{\omega(m)} \sum_{d\le \frac{2x}{mz}} \alpha^{\omega(d)}  \nn
		&\ll  \sum_{\substack{m\le \sqrt{2x/z} \\  m\in[x^{1-u}z^{-4}, 2x^{1-u}z^3] }}  (1-\alpha)^{\omega(m)} \cdot \frac{x(\log x)^{\alpha-1}}{mz} \nn 
		&\ll \frac{x(\log z)(1+ (1-u)\log x)^{-\alpha}(\log x)^{\alpha-1}}{z} .
		\label{eq:bound for S2}
	\end{align}
	Combining \eqref{eq:bound in terms of S1 and S2}, \eqref{eq:bound for S1} and \eqref{eq:bound for S2}, while keeping in mind that $S_1=0$ if $u>2/3$ and that $S_2=0$ if $u<1/3$, completes the proof of the lemma, 
	\end{proof}

	\section{Proof of Lemma \ref{lem:bound of prob of ECprob} using probabilistic techniques}\label{sec:boundary event Prob}
	
	Since $\mathbf Z$ follows the distribution $\text{Dir}(\boldalpha)$, its $i$-th component $Z_i$ follows the $\text{Beta}(\alpha_i, 1 - \alpha_i)$ distribution, meaning that if $[a,b]\subseteq[0,1]$, then 
	\begin{equation}
		\label{eq:Zi distribution}
		\PP\bg[Z_i\in[a,b] \bg] = \frac{1}{\Gamma(\alpha_i)\Gamma(1-\alpha_i)}\int_a^b  \frac{\dee t}{t^{1-\alpha_i} (1-t)^{\alpha_i}} 
			= \frac{\sin(\pi\alpha_i)}{\pi} \int_a^b  \frac{\dee t}{t^{1-\alpha_i} (1-t)^{\alpha_i}} ,
	\end{equation}
	by Euler's reflection formula. 	For this reason, we need the following preliminary estimate. 
	
	\begin{lem}\label{lem:beta estimate}
	Uniformly for $u\in[0,1]$, $\alpha \in (0, 1)$ and $\delta > 0$, we have
		\[
		\sin(\pi\alpha) \int_{[u-\delta,u+\delta]\cap[0,1]} \frac{\dee t}{t^{1-\alpha}(1-t)^\alpha} \ll \frac{\delta}{(u+\delta)^{1-\alpha}(1-u+\delta)^\alpha} .
		\]
		In particular, if $\delta\ge1/\log x$, then 
		\[
		\sin(\pi\alpha) \int_{[u-\delta,u+\delta]\cap[0,1]} \frac{\dee t}{t^{1-\alpha}(1-t)^\alpha} \ll \frac{\delta\log x}{(1+u\log x)^{1-\alpha}(1+(1-u)\log x)^\alpha} .
		\]
	\end{lem}
	
	\begin{proof} Note that both sides of the claimed inequality are invariant under the change of variables $(\alpha,u)\mapsto (1-\alpha,1-u)$. Hence, we may assume without loss of generality that $u\in[0,1/2]$. In addition, observe that
		\[
		\sin(\pi\alpha)\int_0^1 \frac{\dee t}{t^{1-\alpha}(1-t)^\alpha}  = \frac{\sin(\pi\alpha)}{\Gamma(\alpha)\Gamma(1-\alpha)} = \pi. 
		\]
		Hence the lemma is trivially true if $\delta > 1/4$. 
		
		We have thus reduced the proof to the case when $\alpha\in(0,1)$, $u\in[0,1/2]$ and $\delta\in[0,1/4]$. In particular, $u+\delta\le 3/4$, so that $1-t\in[1/4,1]$ for all $t\in[0,u+\delta]$. It thus suffices to prove that 
	\begin{equation}
		\label{eq:beta bound - new goal}
			\sin(\pi\alpha) \int_{[u-\delta,u+\delta]\cap[0,1]} \frac{\dee t}{t^{1-\alpha}} \ll \frac{\delta}{(u+\delta)^{1-\alpha}} .
	\end{equation}

	Assume first that $\delta\le u/2$. We then have $t\ge u/2$ whenever $t\in[u-\delta,u+\delta]$. Using also the trivial bound $\sin(\pi\alpha)\le1$, we conclude that
	\[
	\sin(\pi\alpha) \int_{[u-\delta,u+\delta]\cap[0,1]} \frac{\dee t}{t^{1-\alpha}} \le \int_{[u-\delta,u+\delta]\cap[0,1]} \frac{\dee t}{(u/2)^{1-\alpha}}  \ll \frac{\delta}{(u+\delta)^{1-\alpha}} .
	\]
	This proves the lemma in this case. 
	
	Finally, assume that $\delta\ge u/2$. We then use that
	\[
	\sin(\pi\alpha) \int_{[u-\delta,u+\delta]\cap[0,1]} \frac{\dee t}{t^{1-\alpha}} 
	\le \sin(\pi\alpha)\int_0^{3\delta} \frac{\dee t}{t^{1-\alpha}} 
	=\frac{\sin(\pi \alpha)}{\alpha} \cdot (3\delta)^\alpha 
	\ll \delta^\alpha \asymp \frac{\delta}{(u+\delta)^{1-\alpha}} .
	\]
	This completes the proof of the lemma in all cases.
	\end{proof}
	
	Let us now show Lemma \ref{lem:bound of prob of ECprob}. For simplicity of notation, let us fix $i\in\{1,2,\dots,k\}$ and $u\in(0,1)$, and let us set $\alpha \coloneqq \alpha_i$ and
	\[
	\Delta \coloneqq \frac{1}{(1+ u\log x)^{1-\alpha}(1+(1-u)\log x)^{\alpha}} .
	\]
	Thus, our goal is to show that
	\begin{equation}
		\label{eq:Bprob rewrite}
			\PP\bggg[ \bg| Z_i - u \bg|\le \frac{2\Theta_x}{\log x}\bggg]  \ll \Delta .
	\end{equation}

	Recall that 
	\[
	\Theta_x=\sum_{j\ge1}r(V_j\log x),
	\]
	and that there exists an absolute constant $c\ge1$ such that $r(y)\le c \min\{y,y^{-2}\}$ for all $y>0$ (see \eqref{eq:h(t) estimate}). Since there are at most $10$ indices $j$ such that $V_j\ge 0.1$,  we find that
	\[
	\Theta_x\le \Theta_x'+10c,\quad\text{where}\quad \Theta_x'\coloneqq \sum_{j\ge1:\ V_j<0.1}r(V_j\log x)  .
	\]
	
	Now, using Lemma \ref{lem:beta estimate} and relation \eqref{eq:Zi distribution}, we have that
	\[
	\PP\bggg[ \bg| Z_i - u \bg|\le \frac{200c}{\log x}\bggg]  \ll \Delta .
	\]
	On the other hand, if $\frac{200c}{\log x}<|Z_i-u|\le \frac{2\Theta_x}{\log x}\le \frac{20c+2\Theta_x'}{\log x}$, then we must have $\Theta_x'>90c$. In particular, there exists $m\in\ZZ$ such that $2^m>40c$ and $\Theta_x'\in(2^m,2^{m+1}]$, whence $|Z_i-u|\le \frac{20c + 2^{m+2}}{\log x}$. We thus conclude that
	\[
				\PP\bggg[ \bg| Z_i - u \bg|\le \frac{20c+2\Theta_x'}{\log x}\bggg] \le \sum_{m\in\ZZ:\ 2^m>40c} 
		\PP\bggg[ \bg| Z_i - u_i \bg|\le \frac{20c+2^{m+2}}{\log x},\ \Theta_x>2^m\bggg]  + O(\Delta). 
	\]
	Therefore Lemma \ref{lem:bound of prob of ECprob} will follow if we can prove that
	\begin{equation}
		\label{eq:lem Bprob reduction 1}
		\PP\bggg[ \bg| Z_i - u \bg|\le \frac{20c+4\kappa}{\log x},\ \Theta_x'>\kappa \bggg]  \ll \frac{\Delta}{\kappa} 
		\qquad\text{for all}\ \kappa\ge 40c.
	\end{equation}
	
	We shall make a further reduction. If we set
	\[
	G(\lambda)\coloneqq \#\bg\{j\ge1: V_j\log x\in (\lambda,2\lambda] \bg\},
	\]
	then we have
	\[
	\begin{split}
	\Theta_x' =\sum_{j\ge1:\ V_j<0.1} r(V_j\log x) 
		&\le c\sum_{j\ge1:\ V_j<0.1}\min\bg\{V_i\log x,(V_i\log x)^{-2}\bg\}  \\
		&\le  c\sum_{\substack{m\ge0 \\ 2^m\le 0.1\log x}} \frac{G(2^m)}{4^m} + c \sum_{m<0} 2^{m+1} G(2^m) \\
		&\le 5c \max_{\substack{m\ge0 \\ 2^m\le 0.1\log x}} \bggg(\frac{G(2^m)}{2^{3m/2}} \bggg) + 5c \max_{m<0}\bg(2^{m/2} G(2^m)\bg),
			\end{split}
	\]
	since $\sum_{m\ge0}2^{-m/2} <5$ and $\sum_{m<0}2^{1+m/2}< 5$. We thus find that
	\[
	\begin{split}
		\PP\bggg[ \bg| Z_i - u \bg|\le \frac{20c+4\kappa}{\log x},\ \Theta_x'>\kappa \bggg]  
		&\le \sum_{\substack{m\ge0 \\ 2^m\le 0.1\log x}} \PP\bggg[ \bg| Z_i - u \bg|\le \frac{20c+4\kappa}{\log x},\ G(2^m)> \frac{2^{3m/2} \kappa}{10c}  \bggg]   \\
		&\ +\sum_{m<0} 	\PP\bggg[ \bg| Z_i - u \bg|\le \frac{20c+4\kappa}{\log x},\ G(2^m)> \frac{\kappa 2^{-m/2}}{10c} \bggg]  .
	\end{split}
	\]
	Note that in both sums, we have $G(2^m)\ge 4$, since $\kappa\ge 40c$. Hence, Markov's inequality implies
	\[
	\begin{split}
		\PP\bggg[ \bg| Z_i - u \bg|\le \frac{20c+4\kappa}{\log x},\ \Theta_x'>\kappa \bggg]  
		&\le \sum_{\substack{m\ge0 \\ 2^m\le 0.1\log x}} \frac{100c^2}{\kappa^22^{3m}}
			\EE\bgg[ \one_{| Z_i - u |\le \frac{20c+4\kappa}{\log x} } \cdot \one_{G(2^m) \ge4} \cdot G(2^m)^2\bgg] \\
		&\ +\sum_{m<0} \frac{100c^22^m}{\kappa^2} 
				\EE\bgg[ \one_{| Z_i - u |\le \frac{20c+4\kappa}{\log x} }\cdot   \one_{G(2^m) \ge4}\cdot G(2^m)^2\bgg]  .
	\end{split}
	\]
	
	This reduces \eqref{eq:lem Bprob reduction 1}, and thus Lemma \ref{lem:bound of prob of ECprob}, to proving the following estimate:

	\begin{lem}
		\label{lem:Bprob reduction 2} Uniformly for $\mu\ge1$ and $\lambda\in(0,0.1\log x]$, we have 
			\[
			\EE\bggg[ \one_{| Z_i - u |\le \frac{\mu}{\log x} }  \cdot   \one_{G(\lambda) \ge4} \cdot G(\lambda)^2  \bggg] 
				\ll (\lambda+\mu)\cdot \Delta. 
			\]
	\end{lem}
	
	\begin{proof}  Let us call $E$ the quantity we seek to bound from above. 
		Consider two independent Poisson processes $(A_j)_{j \ge 1}$ with intensity $\alpha\frac{e^{-x}}{x} \dee x$ and $(B_j)_{j \ge 1}$ with intensity $(1-\alpha)\frac{e^{-x}}{x} \dee x$. Hence, the union of the two processes is a new Poisson process  of intensity $\frac{e^{-x}}{x}\dee x$. 
		The sum $S_A \coloneqq \sum_{j\ge1} A_j$ has distribution $\text{Gamma}(\alpha, 1)$ and the sum $S_B \coloneqq \sum_{j\ge1} B_j$ has distribution $\text{Gamma}(1-\alpha, 1)$. If we also set $S\coloneqq S_A+S_B$, then, as we saw in the proof of Proposition \ref{prop:probabilistic dirichlet law}, we have 
		\[
		E
		=\EE\bggg[ \one_{| S_A/S- u |\le \frac{\mu}{\log x} } \cdot \one_{G_A(\lambda)+G_B(\lambda)\ge 4}\cdot  \bgg( G_A(\lambda)+G_B(\lambda) \bgg)^2 \bggg],
		\]
		where $G_A(\lambda) = \sum_{j\ge1} \one_{A_j\log x\in(\lambda S,2\lambda S]}$ and $G_B(\lambda)= \sum_{j\ge1} \one_{B_j\log x\in(\lambda S,2\lambda S]}$. We know $G_A(\lambda)+G_B(\lambda)\ge4$. Hence, if $G_A(\lambda)\ge G_B(\lambda)$, then we must have $G_A(\lambda)\ge 2$ and $G_A(\lambda)\le G_A(\lambda)^2/2$, whence $G_A(\lambda) \le 2\mathop{\sum\sum}_{k>j\ge1} \one_{A_k\log x,\, A_j\log x\in (\lambda S,2\lambda S]}$. So, we find that
		\[
		\begin{split}
		 \bgg( G_A(\lambda)+G_B(\lambda) \bgg)^2 \le 4G_A(\lambda)^2 
		 	&= 4 G_A(\lambda) + 8 \mathop{\sum\sum}_{k>j\ge1} \one_{A_k\log x,\, A_j\log x\in (\lambda S,2\lambda S]}  \\
		 	&\le 16 \mathop{\sum\sum}_{k>j\ge1} \one_{A_k\log x,\, A_j\log x\in (\lambda S,2\lambda S]}  .
			\end{split}
		\]
		The analogous inequality also holds  when $G_B(\lambda)\ge G_A(\lambda)$, with the roles of $A$ and $B$ reversed. We conclude that
		\[
		E
		\le 16 (E_A+E_B),
		\]
		where
		\[
		E_A \coloneqq \EE\bggg[ \one_{|S_A/S - u |\le \frac{\mu}{\log x} }  \mathop{\sum\sum}_{k>j\ge1} \one_{A_k\log x,\, A_j\log x\in (\lambda S,2\lambda S]}  \bggg]
		\]
		and $E_B$ is defined analogously. Using Mecke's equation (cf. Proposition \ref{prel:mecke}), we find that
		\[
		\begin{split}
		E_A &= \alpha^2\iint\limits_{2a_1>a_2>a_1>0}
			\EE\bggg[ \one_{| \frac{a_1+a_2+S_A}{a_1+a_2+S} - u |\le \frac{\mu}{\log x} } \prod_{j=1}^2 \one_{a_j\log x \in (\lambda (a_1+a_2+S),2\lambda (a_1+a_2+S)]} \bggg] \frac{e^{-a_1-a_2}}{a_1a_2}\dee a_1\dee a_2 \\
			&= 	\frac{\alpha^2}{\Gamma(\alpha)\Gamma(1-\alpha)} 
			\iiiint\limits_{  
				\substack{
					2a_1>a_2>a_1>0,\ s_1,s_2>0 \\  
					| \frac{a_1+a_2+s_1}{a_1+a_2+s_1+s_2} - u |\le \frac{\mu}{\log x} \\
					 a_j\log x \in (\lambda (a_1+a_2+s_1+s_2),2\lambda (a_1+a_2+s_1+s_2)] \ (j=1,2) 
				} }
				  \frac{e^{-a_1-a_2-s_1-s_2}}
				  {a_1a_2s_1^{1-\alpha}s_2^\alpha}\dee a_1\dee a_2 \dee s_1\dee s_2,
				\end{split}
		\]
		where we used that the $a_j$'s lie in the same dyadic interval to deduce that $a_2<2a_1$. 		 We make the change of variables $t=s_1/(s_1+s_2)$ and $s=s_1+s_2$. Since $\lambda/\log x\le 0.1$ and $a_1<a_2<2a_1$, the conditions $a_j\log x \in (\lambda (a_1+a_2+s),2\lambda (a_1+a_2+s)]$ for $j=1,2$ imply that $a_j\log x \in(\lambda s,5\lambda s]$. Knowing also that $ |\frac{a_1+a_2+s_1}{a_1+a_2+s} - u |\le \frac{\mu}{\log x}$, we find $|t-u|=|\frac{s_1}{s_1+s_2} - u|\le \frac{10\lambda+\mu}{\log x}$. Finally, we use Euler's reflection formula to write $\Gamma(\alpha)\Gamma(1-\alpha)=\pi/\sin(\pi\alpha)$. We conclude that
		\[
		E_A \le	\frac{\sin(\pi \alpha)}{\pi}
		\iiiint\limits_{  
			\substack{
				a_2>a_1>0,\ s>0,\ t\in[0,1] \\  
				| t-u|\le (10\lambda+\mu)/\log x\\
				a_j\log x \in (\lambda s, 5\lambda s] \ (j=1,2) 
		} }
		\frac{e^{-a_1-a_2-s}}
		{a_1a_2t^{1-\alpha}(1-t)^\alpha}\dee a_1\dee a_2 \dee s\dee t .
		\]
For every fixed value of $s$, the integral over $a_1$ and $a_2$ is $\le(\log 5)^2$, since $\int_w^{5w}\frac{e^{-a}}{a}\dee a \le \log5$ for any $w>0$. We also have $\int_0^\infty e^{-s}\dee s=1$. We thus conclude that
\[
E_A\le \frac{(\log 5)^2\sin(\pi \alpha) }{\pi} \int\limits_{\substack{t\in[0,1] \\ |t-u|\le (10\lambda+\mu)/\log x}} \frac{\dee t}{t^{1-\alpha}(1-t)^\alpha} .
\]
Using Lemma \ref{lem:beta estimate} shows that $E_A\ll (\lambda+\mu)\Delta$. The same estimate holds for $E_B$ too, thus completing the proof of the lemma.
	\end{proof}

	
\newpage

	\clearpage
	\thispagestyle{fancy}
	\fancyhf{} 
	\renewcommand{\headrulewidth}{0cm}
	\lhead[{\scriptsize \thepage}]{}
	\rhead[]{{\scriptsize\thepage}}
	\part{Appendices}
		\appendix

	Because this paper lies in the intersection of number theory and probability theory, readers coming from one of these fields might not be familiar with standard results of the other one. For this reason, we gather here some key results from both fields. We present these results in the following two sections.

	\section{Tools from Number Theory} 
	\label{sec:prel}

	Let $\theta(x) \coloneqq \sum_{p \le x} \log p$ and let $\psi(x) \coloneqq \sum_{n \le x} \Lambda(n)$ where  
	\[
	\Lambda(n) = \begin{cases}
		\log p &\text{if $n = p^k$ for some prime power $p^k$,} \\
		0 &\text{otherwise.}
	\end{cases}
	\]
	Understanding these functions gives information about the distribution of primes. In 1896, Charles de la Vallée Poussin and Jacques Hadamard proved the Prime Number Theorem, thus establishing that $\psi(x)\sim\theta(x)\sim x$ as $x\to\infty$. De la Vallée Poussin proved a more precise version of these estimates, with the error term $O\bg(xe^{-c\sqrt{\log x}}\bg)$ for some positive absolute constant $c$ (see, for example, \cite[Chapter 8]{Koukoulo19}). For the purposes of the present paper, however, the following weaker version of his result suffices:
	
	\begin{prop}[The Prime Number Theorem]
		\label{pnt}
		For $x \ge 2$, we have 
		\[
		\theta(x) = x + \bigoo{\frac{x}{(\log x)^3}}
		\]
		and 
		\[
		\psi(x) = x+\bigoo{\frac{x}{(\log x)^3}}.
		\]
	\end{prop}
	
	We shall also need a stronger version of Mertens's estimate, which in its classical form, yields an estimate for the product $\prod_{p \le x} (1-1/p)$. 
	
	\begin{prop}[Strong Mertens's estimate]
		\label{prop:mertens}
		For $x \ge 2$, we have 
		\[
			\sum_{p^k \le x} \frac{1}{kp^k} = \log_2 x + \gamma + \bigoo{\frac{1}{(\log x)^3}}.
		\]
	\end{prop}

	\begin{proof}
		Using the Prime Number Theorem (Proposition \ref{pnt}), there exists a constant $c_0$ satisfying 
		\[
			\sum_{p^k \le x} \frac{1}{kp^k} = \int_{2^-}^x \frac{\dee \psi(t)}{t\log t} = \log_2 x + c_0 + \bigoo{\frac{1}{(\log x)^3}}.
		\]
		for $x \ge 2$. For the proof of $\gamma = c_0$, see the proof of \cite[Theorem 2.7]{MV07}.
	\end{proof}

	In the proof of Theorem \ref{thm:fact into k parts}, we need the following estimate:
	\begin{prop}
		\label{prel:partial summation bounds}		
		Uniformly for $\alpha\in[0,1]$ and $x,y\ge2$, we have
		\[
			\sum_{n \le x} \alpha^{\omega(n)} \ll x(\log x)^{\alpha-1}
			\quad \text{and} \quad \sum_{x/y<n \le xy} \frac{ \alpha^{\omega(n)}}{n} \ll (\log y)(\log(xy))^{\alpha-1} .
		\]
	\end{prop}

\begin{proof}The first bound follows readily with Theorem 14.2 in \cite{Koukoulo19}. Let us now prove the second one. 
	
When $y\ge \sqrt{x}$, we have
\[
 \sum_{x/y<n \le xy} \frac{ \alpha^{\omega(n)}}{n}\le  \sum_{p|n\ \Rightarrow p \le y^3} \frac{ \alpha^{\omega(n)}}{n} =  \prod_{p\le y^3} \bggg(1+\frac{\alpha}{p-1}\bgg) \ll (\log y)^\alpha 
\]
by the inequality $1+t\le e^t$ and Mertens's theorem (Proposition \ref{prop:mertens}). This completes the proof in this case. Finally, assume that $y\le \sqrt{x}$. We then have
\[
 \sum_{x/y<n \le xy} \frac{ \alpha^{\omega(n)}}{n}
 	\le \sum_{\substack{m\in\ZZ \\ e^m\in[x/y,exy]}}  \sum_{n\in(e^{m-1},e^m]} \frac{ \alpha^{\omega(n)}}{n}
 	\le \sum_{\substack{m\in\ZZ \\ e^m\in[x/y,exy]}} e^{1-m} \sum_{n\le e^m} \alpha^{\omega(n)}.
\]
For every $m$ as above, the innermost sum is $\ll e^m m^{\alpha-1} \asymp e^m (\log(xy))^{\alpha-1}$ by the first part of the lemma and our assumption that $y\le \sqrt{x}$. Since there are $\le1+\log y\ll \log y$ choices for $m$, the needed estimate follows in this last case too.
\end{proof}

\section{Tools from Probability Theory}

	\subsection{The Total Variation Distance}
	
	The \textit{total variation distance} is a metric between two probability distributions. Let $\mu$ and $\nu$ be two probability measures on $\CC$. Then
	\begin{equation}
		\label{eq:dtv dfn}
		\dtv(\mu, \nu) \coloneqq \sup \left|\mu(A) - \nu(A)\right|,
	\end{equation}
	where the supremum is taken over all Lebesgue-measurable subsets of $\CC$. For any real number $a$, let
	\[
		a^+ \coloneqq \max\{a, 0\} \quad \text{and} \quad a^- \coloneqq \max\{-a, 0\}.
	\]When $\mu$ and $\nu$ are supported on $\ZZ_{\ge 1}$, here are some alternative definitions of the total variation distance:
	\begin{lem}
		\label{lem:equiv-def-dtv}
		Let $\mu$ and $\nu$ be two probability measure supported on $\NN$. Then
		\[
			\dtv(\mu, \nu) = \sum_{i \ge 1} (\mu(i) - \nu(i))^+ = \sum_{i \ge 1} (\mu(i) - \nu(i))^-
		\]
	\end{lem}
	
	\begin{proof}
		Let $E \coloneqq \{i \in \NN \, : \, \mu(i) > \nu(i)\}$ and let $\xi\coloneqq \mu - \nu$. Note that for any $B\subseteq \NN$, we have $\xi(B\cap E) \ge 0 \ge \xi(B\cap E^c)$. Therefore, for any $A \subseteq \NN$, we have
		\[
		\xi(A) = \xi(E) + \xi(A \cap E^c) - \xi(A^c \cap E) \le \xi(E).
		\]
		and 
		\[
		\xi(A) = \xi(E^c) + \xi(A \cap E) - \xi(A^c \cap E^c) \ge \xi(E^c)
		\]
		Therefore, $\dtv(\mu, \nu) = \sup_{A \subseteq \NN} |\xi(A)| = \max\{\xi(E), -\xi(E^c)\}$. Note that $\xi(E) = \sum_{i \ge 1} (\mu(i) - \nu(i))^+$, that $-\xi(E^c) = \sum_{i \ge 1} (\mu(i) - \nu(i))^-$ and that
		\[
		\sum_{i \ge 1} (\mu(i) - \nu(i))^+ = \sum_{i \ge 1} \bg[(\mu(i) - \nu(i)) + (\mu(i) - \nu(i))^-\bg] = \sum_{i \ge 1} (\mu(i) - \nu(i))^-.
		\]
		The lemma follows.
	\end{proof}
	
	The total variation distance will be especially useful in this paper because of the following proposition. In \cite[Section 3.8]{Arratia02}, Arratia proved that for any two random variables $X$ and $Y$ returning positive integers, we can always construct $X'$ and $Y'$ within the same probability space such that:
	\begin{itemize}
		\item $X'$ and $Y'$ have the same distribution as $X$ and $Y$, respectively;
		
		\item $Y'$ is a function of $X', U, V$ where $(U, V)$ is a point uniformly chosen in the unit square independent of $X'$;
		
		\item $\PP[X' \ne Y'] = \dtv(X, Y)$. 
	\end{itemize}
	
	We repeat his proof here. Let $\mu$ and $\nu$ be two probability measures supported on $\NN$. Let $z_j \coloneqq \sum_{i \le j} \frac{(\mu(i) - \nu(i))^-}{\dtv(\mu, \nu)}$ (with $z_0 \coloneqq 0$). We consider the function $f_{\mu, \nu}\colon \NN \times (0, 1)^2 \to \NN$ defined as
	\[
		f_{\mu, \nu}(m; a, b) \coloneqq \begin{cases}
				m &\text{if $a \cdot \mu(m) \le \nu(m)$,} \\
				\sum_{i \ge 1} i \cdot \one_{z_{i-1} < b \le z_i} &\text{otherwise.}
			\end{cases}
	\]
		This is the function used in Section \ref{sec:coupling} for the extraction of $N_x$. 
	Note that if $a,b\in(0,1)$ and $m\in\NN$, then we have the equivalency 
	\begin{equation}
		\label{eq:f TV property}
		f_{\mu, \nu}(m; a, b) \ne m\quad\iff\quad a\cdot \mu(m) > \nu(m).
	\end{equation}
	Indeed, the direction ``$\Rightarrow$'' is obvious. To see the converse direction, note that if $a\cdot\mu(m)>\nu(m)$, then $(\mu(m)-\nu(m))^-=0$, and hence the interval $(z_{m-1},z_m]$ is empty. 

	\begin{lem}[Arratia, \cite{Arratia02}]
		\label{lem:coupling-dtv}
		Let $\mu$ and $\nu$ be two probability measures supported on $\NN$, let $X$ be a random variable with law $\mu$, and let $U$ and $U'$ be two uniform random variables in $(0, 1)$ such that $X, U, U'$ are independent. Let $Y \coloneqq f_{\mu, \nu}(X; U, U')$ with $f_{\mu, \nu}$ defined as above. Then $\PP[X \ne Y] = \dtv(\mu, \nu)$ and $\PP[Y \in A] = \nu(A)$.
	\end{lem}
	
	\begin{proof}
		Using \eqref{eq:f TV property}, we find that 
		\begin{equation}
			\label{eq:event X neq Y}
			\{X \ne Y\}=\{U \cdot \mu(X) > \nu(X)\}. 
		\end{equation}
		Furthermore, we directly compute that 
		\begin{equation}
			\label{eq:prob X eq m and U large}
			\PP\bg[U\cdot \mu(m)> \nu(m), X = m\bg] = (\mu(m) - \nu(m))^+.
		\end{equation}
		Therefore,
		\begin{equation}
			\label{eq:prob X neq Y}
			\PP[X \ne Y] = \PP[U \cdot \mu(X) > \nu(X)] = \sum_{m \ge 1} (\mu(m) - \nu(m))^+ = \dtv(\mu, \nu)
		\end{equation}
		with Lemma \ref{lem:equiv-def-dtv}.
		
		Next, we prove that $\PP[Y = n] = \nu(n)$ for any $n \in \NN$. Note that $Y = n$ if, and only if, one of two disjoint events happen: either we have $U \cdot \mu(n)\le \nu(n)$ with $X = n$, or we have $U\cdot \mu(X) > \nu(X)$ with $z_{n-1} < U' \le z_n$. Therefore, with \eqref{eq:prob X eq m and U large} and \eqref{eq:prob X neq Y}, we have
		\begin{align*}
			\PP[Y = n] 
			&= \PP\bg[U \cdot \mu(n)\le \nu(n), X = n\bg]  + \PP\bg[U\cdot \mu(X) > \nu(X)\bg] \cdot \PP\bg[z_{n-1} < U' \le z_n\bg] \\
			&= \bgg(\mu(n) - \bg(\mu(n) - \nu(n)\bg)^+\bgg) + \dtv(\mu, \nu) \cdot \frac{(\mu(n) - \nu(n))^-}{\dtv(\mu, \nu)} = \nu(n),
		\end{align*}
		where we used \eqref{eq:event X neq Y} and \eqref{eq:prob X neq Y} to show that $\PP\bg[U\cdot \mu(X) > \nu(X)\bg]=\dtv(\mu,\nu)$. 
		This concludes the proof of the lemma.
	\end{proof}

	\subsection{Poisson Processes}
	
	The following definition and propositions are borrowed from Kingman's book on Poisson processes \cite{Kingman93}. Let $(S, \mathcal S)$ be a measurable space with $S$ being a subset of $\RR^d$ for some $d \ge 1$. A \textit{Poisson process} on a state space $S$ with mean measure $\mu$ is a random countable subset $\Pi \subseteq S$ such that:
	\begin{itemize}
		\item for any disjoint measurable subsets $A_1, \ldots, A_n$ of $S$, the random variables $\#(\Pi \cap A_1)$, \dots, $\#(\Pi \cap A_n)$ are independent;
		
		\item the random variable $\#(\Pi \cap A)$ is a Poisson random variable of parameter $\mu(A)$ for any $A\subset S$ measurable.
	\end{itemize}
	We can see $\Pi$ as an element of the measurable space $(\Omega_S, \mathcal F_S)$ where $\Omega_S$ is the set of countable subsets of $S$ and $\mathcal F_S$ is the smallest $\sigma$-algebra for which the map $\Pi \mapsto \#\{\Pi \cap B\}$ is measurable for all $B \in \mathcal S$. If $\mu$ has no atoms, meaning no singleton with positive probability, and is $\sigma$-finite, meaning that $S$ is a countable union of measurable sets with finite measure, then a Poisson process with mean measure $\mu$ always exists (see \cite{Kingman93}, the Existence Theorem in section 2.5 for a proof). If $\mu$ is absolutely continuous with respect to the Lebesgue measure, then the function $\lambda \colon S \to \RR_{\ge 0}$ such that $\mu(A) = \riemint{A}{}{\lambda(x)}{x}$ for all measurable subsets $A \subseteq S$ is called the \textit{intensity} of the Poisson process. Here are a few important propositions about Poisson processes that we will use in the paper and proved in \cite{Kingman93}:
	\begin{prop}[Mapping Theorem, \cite{Kingman93} section 2.3]
		\label{prel:map}
		If $\Pi$ is a Poisson process with mean measure $\mu$ on $S$, and $f:S \to T$ is a measurable function such that $\mu^*(B) \coloneqq \mu(f^{-1}(B))$ has no atoms, then $f(\Pi)$ is a Poisson process on $T$ with measure $\mu^*$.
	\end{prop}
	
	\begin{prop}[Colouring Theorem, \cite{Kingman93} section 5.1]
		\label{prel:colour}
		If $\Pi$ is a Poisson process with mean measure $\mu$ on $S$, and the points are randomly coloured with $k$ colours such that the probability of a point receiving the colour $i$ is $p_i$, and such that the colour of a point is independent of different points and of the position of the point. Let $\Pi_i$ be the subset of $\Pi$ with colour $i$. Then all the $\Pi_i$ are independent Poisson processes with mean measures $\mu_i = p_i\mu$.
	\end{prop}

	\begin{prop}[Campbell's Theorem, \cite{Kingman93} section 3.2]
		\label{prel:campbell}
		Let $\Pi$ be any Poisson process on $S$ with mean measure $\mu$. Let $f\colon S \to \RR$ be a measurable function. Then 
		\[
		\Sigma = \sum_{X \in \Pi} f(X)
		\]
		is absolutely convergent almost surely if, and only if, 
		\[
		\riemint{S}{}{\min\{f, 1\}}{\mu} < \infty.
		\]
		If this condition holds, then
		\[
		\EE[e^{s\Sigma}] = \exp\left(\riemint{S}{}{(e^{sf} - 1)}{\mu}\right)
		\]
		for any complex $s$ for which the integral converges. Moreover,
		\[
		\EE[\Sigma] = \riemint{S}{}{f}{\mu}
		\]
		if the integral converges. In the case where it converges, we also have
		\[
		\vard[\Sigma] = \riemint{S}{}{f^2}{\mu}.
		\]
	\end{prop}
	
	Many probabilities or expectations that involve Poisson processes in this paper can be reformulated as 
	\[
	\EE \sum_{X \in \Pi} f(\Pi\setminus \{X\}, X).
	\]
	In these cases, there is a generalization of the formula for $\EE[\Sigma]$ in Campbell's Theorem, called the Mecke equation, allowing us to compute these objects:
	\begin{prop}[Mecke equation, \cite{LastPenrose18} Theorem 4.5]
		\label{prel:mecke}
		Let $\Pi$ be a Poisson process on $S$ with a $\sigma$-finite mean measure $\mu$, and let $f: \Omega_S \times S^k \to [0, \infty)$ be measurable. Then we have
		\[
		\EE \sum_{\substack{X_1, \ldots, X_k \\ \text{all distinct}}} f(\Pi\setminus \{X_1, \ldots, X_k\}; X_1, \ldots, X_k) = \int_S \cdots \int_S \EE[f(\Pi; x_1, \ldots, x_k)] \, \mathrm d\mu(x_1)\, \cdots\, \mathrm d\mu(x_k).
		\]
	\end{prop}

\end{document}